\pdfoutput=1
\documentclass{elsarticle}

\usepackage{lineno}
\modulolinenumbers[5]
\journal{\ldots}

\usepackage{color}
\usepackage{amsmath}
\usepackage{graphicx}
\usepackage{amssymb}
\usepackage{latexsym}
\usepackage[a4paper]{geometry}

\usepackage{hyperref}
\definecolor{myblue}{rgb}{0,0,0.6}
\hypersetup{colorlinks=true,linkcolor=myblue,citecolor=myblue,filecolor=myblue,urlcolor=myblue, bookmarksopen=true, bookmarksopenlevel=3}

\allowdisplaybreaks[4]

\newtheorem{theorem}{Theorem}[section]
\newtheorem{corollary}[theorem]{Corollary}
\newtheorem{definition}[theorem]{Definition}
\newenvironment{proof}[1][Proof]{\noindent \emph{#1.} }
{\hfill \ \rule{0.5em}{0.5em}}
\newtheorem{lemma}[theorem]{Lemma}
\newtheorem{proposition}[theorem]{Proposition}
\newtheorem{assumption}[theorem]{Assumption}

\newtheorem{remark}[theorem]{Remark}
\numberwithin{equation}{section}

\newcommand{\noi}{\noindent}
\newcommand{\R}{\mathbb{R}}
\newcommand{\cL}{{\cal L}}
\newcommand{\cK}{{\cal K}}
\newcommand{\cV}{{\cal V}}
\newcommand{\cM}{{\cal M}}
\newcommand{\cO}{{\cal O}}
\newcommand{\bx}{\mathbf{x}}
\newcommand{\br}{\mathbf{r}}
\newcommand{\bze}{\mathbf{0}}
\newcommand{\bn}{\mathbf{n}}
\newcommand{\bv}{\mathbf{v}}
\newcommand{\bw}{\mathbf{w}}
\newcommand{\bW}{\mathbf{w}}
\newcommand{\bg}{\mathbf{g}}
\newcommand{\bu}{\mathbf{u}}

\newcommand{\ba}{\mathbf{a}}
\newcommand{\C}{\mathbb{C}}
\newcommand{\la}{\lambda}
\newcommand{\eps}{\varepsilon}
\newcommand{\re}{{\rm e}}
\newcommand{\ri}{{\rm i}}
\newcommand{\rd}{{\rm d}}
\newcommand{\Rea}{\mathbb{R}}
\newcommand{\Com}{\mathbb{C}}
\newcommand{\half}{\frac{1}{2}}

\newcommand{\pdiff}[2]{\frac{\partial #1}{\partial #2}}
\newcommand{\po}{\Gamma}

\newcommand{\bLv}{\overline{\cL v}}
\newcommand{\bgv}{\overline{\gv}}
\newcommand{\dive}{\nabla \cdot}
\newcommand{\length}{{L}}
\newcommand{\dvdn}{\pdiff{v}{n}}
\newcommand{\dnu}{\partial_n u}
\newcommand{\dnv}{\partial_n v}
\newcommand{\nT}{\nabla_\Gamma}
\newcommand{\gu}{\nabla u}
\newcommand{\gv}{\nabla v}
\newcommand{\gvb}{\overline{\nabla v}}
\newcommand{\vb}{\overline{v}}
\newcommand{\GammaN}{\Gamma}
\newcommand{\LtG}{{L^2(\Gamma)}}
\newcommand{\HokO}{{H^1_k(\Omega)}}
\newcommand{\LtO}{{L^2(\Omega)}}
\newcommand{\HoO}{H^1(\Omega)}
\newcommand{\tendi}{\rightarrow \infty}
\newcommand{\tendo}{\rightarrow 0}
\newcommand*{\N}[1]{\left\|#1\right\|}

\newcommand{\tfa}{\text{ for all }}

\newcommand{\tas}{\text{ as }}
\newcommand{\tand}{\text{ and }}

\newcommand{\cont}{{\rm cont}}
\newcommand{\coer}{{\rm coer}}
\newcommand{\Ccont}{C_{\cont}}
\newcommand{\Ccoer}{C_{\coer}}
\newcommand{\Ccontj}{C_{\cont}^{(j)}}
\newcommand{\Ccoerj}{C_{\coer}^{(j)}}
\newcommand{\bV}{\mathbf{v}}
\newcommand{\matrixstyle}[1]{{\mathsf{#1}}}
\newcommand{\matrixS}{\matrixstyle{S}}
\newcommand{\matrixL}{\matrixstyle{ L}}
\newcommand{\matrixM}{\matrixstyle{ M}}
\newcommand{\matrixB}{\matrixstyle{ B}}
\newcommand{\matrixC}{\matrixstyle{ C}}
\newcommand{\matrixD}{\matrixstyle{ D}}
\newcommand{\matrixI}{\matrixstyle{ I}}
\newcommand{\matrixNo}{\matrixstyle{ N}^{(0)}}
\newcommand{\matrixNone}{\matrixstyle{ N}^{(1)}}
\newcommand{\matrixNtwo}{\matrixstyle{ N}^{(2)}}
\newcommand{\matrixA}{\matrixstyle A}
\newcommand{\matrixP}{{\matrixstyle P}}
\newcommand{\bfx}{\mathbf{x}}
\newcommand{\bfd}{\mathbf{d}}
\newcommand{\bfr}{\mathbf{r}}

\begin{document}
\begin{frontmatter}
\title{Can coercive formulations lead to fast and accurate solution of the Helmholtz equation?}
\author[UCL]{Ganesh C.\ Diwan}
\ead{g.diwan@ucl.ac.uk}

\author[Pv]{Andrea Moiola\corref{cor1}}
\ead{andrea.moiola@unipv.it}
\cortext[cor1]{Corresponding author}

\author[Bath]{Euan A.\ Spence}
\ead{E.A.Spence@bath.ac.uk}

\address[UCL]{Department of Medical Physics \& Biomedical Engineering, University College London, London, WC1E 6BT, UK}
\address[Pv]{Department of Mathematics, University of Pavia, 27100 Pavia, Italy}
\address[Bath]{Department of Mathematical Sciences, University of Bath, Bath, BA2 7AY, UK}

\begin{abstract}
A new, coercive formulation of the Helmholtz equation was introduced in \cite{MoSp:14}. In this paper we investigate $h$-version Galerkin discretisations of this formulation, and the iterative solution of the resulting linear systems.
We find that the coercive formulation behaves similarly to the standard formulation in terms of the pollution effect (i.e.~to maintain accuracy as $k\to\infty$, $h$ must decrease with $k$ at the same rate as for the standard formulation).
We prove $k$-explicit bounds on the number of GMRES iterations required to solve the linear system of the new formulation when it is preconditioned with a prescribed symmetric positive-definite matrix. 
Even though the number of iterations grows with $k$, these are the first such rigorous bounds on the number of GMRES iterations for a preconditioned formulation of the Helmholtz equation, where the preconditioner is a symmetric positive-definite matrix.  
\end{abstract}


\end{frontmatter}

\section{Introduction: the goals of this paper}\label{sec:goals}

The Helmholtz equation $\Delta u +k^2 u=0$ is difficult to solve numerically, when the wavenumber $k$ is large, for the following three reasons:
\begin{enumerate}
\item 
The solutions of the homogeneous Helmholtz equation oscillate on a scale of $1/k$, and so to approximate them accurately with piecewise polynomial functions (e.g.\ using the finite element method) one needs the total number of degrees of freedom, $N$, to be proportional to $k^d$ as $k$ increases, $d\in\mathbb N$ being the spatial dimension.
\item The \emph{pollution effect} means that for fixed-order finite-element methods with $N\sim k^d$, even though the best-approximation error is bounded independently of $k$, the relative error grows with $k$. 
The fact that $N\gg k^d$ is required for the relative error to be bounded independently of $k$ leads to very large matrices, and hence to large (and sometimes intractable) computational costs.
\item 
The standard variational formulation of the Helmholtz equation is not coercive (i.e.~it is sign-indefinite) when $k$ is sufficiently large; in other words, zero is in the \emph{numerical range} or \emph{field of values} of the operator (see Definition \ref{def:fov} below). This indefiniteness is inherited by the Galerkin linear system; therefore 
even when the linear system has a unique solution (which depends on the discretisation and on $k$), one expects iterative methods to behave extremely badly if the system is not preconditioned.
\end{enumerate}

A new  formulation of the Helmholtz equation was introduced in \cite{MoSp:14} (see the recap in \S\ref{sec:coercive} below); the advantage of this new formulation is that the sesquilinear form is continuous and coercive for all $k>0$,  and thus this formulation does not suffer from the third difficulty above. The disadvantage is that it is posed in a subset of $H^1(\Omega)$, namely the space $V$ defined by \eqref{eq:V}, and conforming discretisations of $V$ require $C^1$ elements (like conforming discretisations of the standard least-squares formulation).

The goals of this paper are to answer the following two questions for  the $h$-version of the Galerkin method applied to the formulation of \cite{MoSp:14} (defined by \eqref{eq:coercive} below):
\begin{enumerate}
\item[Q1.] How must $h$ decrease with $k$ for the relative error to be below a prescribed ($k$-independent) accuracy as $k\tendi$? 
(See Definition \ref{def:hka} below for a more precise description of this property.)
\item[Q2.] How does the number of GMRES iterations grow with $k$? 
\end{enumerate}
We then compare the answers with the corresponding answers for the the standard variational formulation \eqref{eq:vfH1} and the least-squares formulation \eqref{eq:vfLS}.

We discretise all three formulations with $C^1$ elements; this is necessary for the formulation of \cite{MoSp:14} and the least-squares formulation because they are posed in the space $V$ \eqref{eq:V}. The standard variational formulation only requires $C^0$ elements, but we use $C^1$ elements to keep the comparison uniform across formulations.

Q1 is investigated in  \S\ref{sec:test_acc}, Q2 is investigated in \S\ref{sec:test_fast}, and the results are combined in \S\ref{sec:conc} to give $k$-explicit estimates on the finite-element error of the approximation to the Galerkin solution computed with a number of GMRES iterations whose $k$-dependence is given explicitly.

We highlight the complementary investigation of the formulation of \cite{MoSp:14} by Ganesh and Morgenstern in \cite{GaMo:17a} (and their generalisation of the formulation to problems with variable refractive index in \cite{GaMo:17b}). 
We compare the results of \cite{GaMo:17a, GaMo:17b} to our results in \S\ref{rem:GM1} below.

\section{Definitions and existing theory of the three variational formulations considered}

In this paper, we consider the model Helmholtz problem of the interior impedance problem, and we are particularly interested in the case that $k L$ is large, $L$ being a characteristic length of the computational domain. 

\begin{definition}[Interior Impedance Problem (IIP)]\label{def:IIP}
Let $\Omega\subset \Rea^d$, $d \geq 1$ be a bounded Lipschitz open set and let $\Gamma:= \partial \Omega$.
Given $f\in \LtO$, $g \in \LtG$, and $k>0$ find $u\in H^1(\Omega)$ such that
\begin{subequations}\label{eq:bvp_main}\vspace{-5mm}
\begin{align}
\cL u :=\Delta u + k^2 u &= -f \quad \mbox{ in } \Omega,\label{eq:bvp_main1}\\
\dnu -\ri k  u&= g \quad \mbox{ on } \Gamma,\label{eq:bvp_main2}
\end{align}
\end{subequations}
where $\partial_n$ denotes the normal derivative operator (see, e.g., \cite[Lemma 4.3]{Mc:00}).
\end{definition}

Since the fundamental solution of the operator $\cL$ is known explicitly, 
the IIP can be solved by boundary integral equations, which have the advantage that the dimension of the problem is reduced.
Nevertheless, there is large interest in the numerical solution of the IIP via discretisations in the domain (as opposed to on the boundary), partly motivated by the large interest in the heterogeneous Helmholtz equation $\Delta u+ k^2 n u=0$, where $n$ is a function of position; boundary-integral-equation techniques are no longer applicable to this latter equation since there does not exist an explicit expression for the fundamental solution.

\subsection{Recap of the theory of the standard variational formulation}

The standard variational formulation of the IIP is formed by multiplying the PDE \eqref{eq:bvp_main1} by a test function $\overline{v}$ and integrating by parts (i.e.~using Green's theorem).

\begin{definition}[Standard variational formulation in $H^1$]\label{def:standard}
Given $f\in \LtO, g\in \LtG$, and $k>0$, find $u\in H^1(\Omega)$ such that
\begin{equation}
a_{ST}(u,v)= F_{ST}(v) \quad\tfa v\in H^1(\Omega),
\label{eq:vfH1}
\end{equation}
$$\text{where}\quad
a_{ST}(u,v):=\int_\Omega\left( \gu \cdot\gvb - k^2 u\vb\right) \rd \bx - \ri k \int_\Gamma  \,u \vb \,\rd s
\quad \text{and}\quad
F_{ST}(v):= \int_\Omega f \vb \,\rd \bx + \int_\Gamma g \vb\, \rd s.
$$
\end{definition}
Given a finite-dimensional subspace $H_N\subset H^1(\Omega)$, the \emph{Galerkin method} is, 
\begin{equation}\label{eq:Galerkin}
\text{ find } u_N \in H_N \text{ such that } \quad a_{ST}(u_N ,v_N) = F_{ST}(v_N) \quad\tfa v_N\in H_N.
\end{equation}

In this paper we consider the $h$-version of the finite element method ($h$-FEM); i.e.~we consider a sequence $(H_N)_{N\in \mathbb{Z}}$ of finite-dimensional nested subspaces, with each $H_N$ a space of piecewise polynomials of some fixed degree $p\geq 0$ and mesh diameter $h$, so that the subspace dimension $N$ (i.e.~the total number of degrees of freedom) satisfies $N\sim h^{-d}$. 
We highlight that there are many other discretisations schemes for the Helmholtz equation; we touch on some of these below (e.g.~$hp$-FEM in \S\ref{sec:accurate}, Trefftz methods in \S\ref{sec:coercive}), but remained focused on the $h$-FEM because of its wide and sustained use by people interested in solving the Helmholtz equation in applications.

If $\{\phi_i: i = 1, \ldots ,   N \}$ is a (real) basis of $H_N$, then the Galerkin equations \eqref{eq:Galerkin} are equivalent to the $N$-dimensional linear system 
\begin{equation}\label{eq:discrete}
\matrixA \bu  = \mathbf{f} , 
\quad 
\text{with} 
\quad 
\matrixA := \matrixS -k^2 \matrixM - \ri k \matrixNo ,  
\end{equation}
where $\matrixS_{\ell,m} = \int_{\Omega} \nabla \phi_\ell \cdot \nabla
\phi_m \, \rd \bx$ 
is the stiffness matrix, $\matrixM_{\ell,m} = \int_{\Omega}
\phi_\ell  \phi_m \, \rd \bx$ is the mass matrix,   and
$\matrixNo_{\ell,m} = \int_{\Gamma} \phi_\ell  \phi_m\, \rd s$ is the boundary
mass matrix. 
Note that $\matrixA$ is symmetric but not Hermitian.

Throughout the paper, we use the notation $a\lesssim b$ to mean that there exists a $C>0$, independent of $h$ and $k$ such that $a\leq Cb$. We write $a\sim b$ when $a\lesssim b$ and $b\lesssim a$.

\subsubsection{First numerical-analysis goal: accuracy of Galerkin solutions.}\label{sec:accurate}
The standard numerical analysis of the $h$-FEM applied to elliptic PDEs is concerned with the limit $h\tendo$ with other parameters, such as $k$, fixed. 
When solving wave problems such as the IIP, it is natural to consider $h$ as a function of $k$, and have the goal that relative error is controlled to a prescribed accuracy uniformly in $k$. The following (non-standard) definition will make it easier to refer to this property in the rest of the paper. Recall that the natural norm on $H^1(\Omega)$ for Helmholtz problems is given by 
\begin{equation}\label{eq:H1k}
\N{v}_{H^1_k(\Omega)}^2:=\N{\nabla v}^2_{\LtO}+k^2\N{v}^2_{\LtO};
\end{equation}
the rationale behind this weighting is that, if $u$ satisfies $\Delta u +k^2 u=0$, one expects that $\|\gu\|_{L^2(\Omega)} \sim k\|u\|_{L^2(\Omega)}$, under which both terms in the norm \eqref{eq:H1k} are of the same magnitude (see Remark \ref{rem:ass1}).

\begin{definition}[$hk^a$-accurate]\label{def:hka}
Given $a>0$, we say that an $h$-FEM for the IIP is \emph{$hk^a$-accurate} if
given $0<\epsilon<1$ and $k_0>0$ there exists $C=C(\epsilon,k_0)$ such that if
$hk^a \leq C$,
then the sequence of Galerkin solutions $u_N$ satisfies
\begin{equation}\label{eq:rel_error}
\frac{\N{u-u_N}_{\HokO}}{\N{u}_{\HokO}}\leq \epsilon
\qquad
\text{for all}\; k\geq k_0.
\end{equation}
\end{definition}

The question of for what $a>0$ the standard FEM is $hk^a$-accurate
was thoroughly investigated by Ihlenburg and Babu\v{s}ka in 1-d \cite{IhBa:95a}, \cite{IhBa:97} (following earlier work by Bayliss, Goldstein, and Turkel \cite{BaGoTu:85}).
With $H^1$-conforming piecewise-polynomial subspaces of order $p\geq 1$, Ihlenburg and Babu\v{s}ka showed that when $p=1$, the $h$-FEM is $hk^{3/2}$-accurate (assuming $u\in H^2$) \cite[Equation 3.25]{IhBa:95}, \cite[Equation 4.5.15]{Ih:98} with numerical experiments indicating that this is sharp \cite[Figure 11]{IhBa:95a}, \cite[Figure 4.13]{Ih:98}.
Furthermore they showed that when $p\geq 2$ the $h$-FEM is $hk^{(2p+1)/(2p)}$-accurate (assuming $u\in H^{p+1}(\Omega)$ and $f\in H^{p-1}(\Omega)$) \cite[Corollary 3.2]{IhBa:97}, \cite[Theorem 4.27 and Equation 4.7.41]{Ih:98}. 

The situation for $d=2,3$ is less understood. Numerical experiments (e.g., \cite[\S3]{BaGoTu:85}) indicate that, when $p=1$, 
the $h$-FEM is $hk^{3/2}$-accurate, but this has yet to be proved. 
If $\Omega$ is a convex polygon (for $d=2$) or polyhedron (for $d=3$), 
Wu proved in \cite{Wu:13} that,  if $hk^{3/2}$ is sufficiently small,~then 
\begin{equation}\label{eq:Wu}
\|u-u_N\|_{\HokO} \lesssim \|f\|_{L^2(\Omega)}+ \|g\|_{H^{1/2}(\Gamma)}. 
\end{equation}
In the case when $\Omega$ is star-shaped with respect to a ball and $\Gamma$ is analytic, Melenk and Sauter proved that \eqref{eq:Wu} holds when $hk^{(p+1)/p}$ is sufficiently small \cite[Equation (5.14b)]{MeSa:11}.
\footnote{Melenk and Sauter also proved that the $hp$-FEM is quasi-optimal (i.e.~\eqref{eq:qo} holds) when $kh/p$ is sufficiently small and $p\geq C\log k$ for some sufficiently large $C$ \cite[Theorem 5.8]{MeSa:11}, and Esterhazy and Melenk proved analogous results for polygons \cite[Theorem 4.2]{EsMe:12}.}

\begin{remark}[Quasi-optimality]\label{rem:qo}
A related goal to \eqref{eq:rel_error} is for the $h$-FEM to be \emph{quasi-optimal}: 
\begin{equation}\label{eq:qo}
\N{u-u_N}_{\HokO} \lesssim \min_{v_N\in H_N} \N{u-v_N}_{\HokO}.
\end{equation}
In 1-d, Ihlenburg and Babu\v{s}ka proved that \eqref{eq:qo} holds for $p=1$ when $hk^2$ is sufficiently small \cite[Theorem 3]{IhBa:95}, \cite[Theorems 4.9 and 4.13]{Ih:98}, and numerical experiments indicated that this is sharp \cite[Figures 7 and 8]{IhBa:95}, \cite[Figure 4.11]{Ih:98}.
In 2- and 3-d, Melenk proved that \eqref{eq:qo} holds when $hk^2$ is sufficiently small \cite[Proposition 8.2.7]{Me:95}, under the a priori estimate \eqref{eq:bound} below and assuming that $u\in H^2(\Omega)$.
\end{remark}

\begin{remark}[The pollution effect]
The pollution effect can \emph{either} be defined by saying that a numerical method suffers the pollution effect if the condition ``$hk$ sufficiently small'' is not enough to ensure that the relative error is bounded independently of $k$ (i.e.~\eqref{eq:rel_error}), see \cite[\S4.6.1]{Ih:98}, \emph{or} by saying that a numerical method suffers the pollution effect if the condition ``$hk$ sufficiently small'' is not enough to ensure $k$-independent quasi-optimality (i.e.~\eqref{eq:qo}), see \cite[Definition 2.1]{BaSa:00}. 

We now show that if $k$-independent quasi-optimality holds, then the relative error decreases with $k$,
and thus the second definition of the pollution effect is stronger than the first. 
For continuous piecewise-polynomial elements on a simplicial mesh and $w\in H^2(\Omega)$ we have 
\begin{equation*}
\min_{v_N \in V_N} \N{w-v_N}_{\HokO} \lesssim h  \N{w}_{H^2(\Omega)} + hk \N{w}_{H^1(\Omega)}
\end{equation*}
by, e.g., properties of the quasi-interpolant given in \cite[Theorem 4.1]{ScZh:90}. Furthermore, 
assuming that derivatives of Helmholtz solutions scale with $k$, see Assumption \ref{ass:1} (in particular \eqref{eq:1} with $m=1$) below, 
if quasi-optimality \eqref{eq:qo} holds, we have that 
$
\N{u-u_N}_{\HokO}  \lesssim hk \N{u}_{\HokO}.
$
As recalled above, quasi-optimality holds for $hk^2$ sufficiently small, and thus the relative error then decreases like $1/k$ as $k$ increases in this case.
\end{remark}

\subsubsection{Second numerical-analysis goal: rapid solution of linear system}\label{sec:fast}

From \S\ref{sec:accurate}, the dimension $N$ of the Galerkin method $A$ must grow at least like $k^d$ as $k$ increases, which puts 3-d large-$k$ problems out of range of direct solvers. 
The Galerkin matrix $A$ \eqref{eq:discrete} is non-Hermitian, and in general it is nonnormal. General iterative methods such as preconditioned (F)GMRES therefore have to be employed for the solution of the linear system \eqref{eq:discrete}.

Without preconditioning, GMRES performs badly when applied to Helmholtz problems with $k$ large, and 
the search for good preconditioners for Helmholtz problems is therefore a topic of much current interest; see, e.g., the reviews \cite{ErGa:12}, \cite{GaZh:16} and the references therein. 
One is the reasons this is difficult 
is that analysing the convergence of (preconditioned) GMRES is hard, because an analysis of the spectrum of the system matrix alone is not sufficient for any rigorous convergence estimates. In \S\ref{sec:GMRES} we recap the existing tools based on the \emph{field of values}/\emph{numerical range}.

\begin{definition}[Field of values/numerical range]\label{def:fov}
Given an $N\times N$ complex matrix $\matrixC$, the \emph{field of values}/\emph{numerical range} of $\matrixC$ (in the Euclidean inner-product $\langle\cdot,\cdot\rangle_2$), $W(\matrixC)$, is defined by 
$$
W(\matrixC):= \big\{ \langle \matrixC \bv, \bv\rangle_2 : \bv \in \Com^N, \|\bv\|_2=1\big\}.
$$
\end{definition}

\subsubsection{The role of coercivity.}
Two key properties of sesquilinear forms, such as $a_{ST}(\cdot,\cdot)$ are \emph{continuity} and \emph{coercivity}. 
Indeed, given a sesquilinear form $a(\cdot,\cdot)$ on a Hilbert space $\cV$ with norm $\|\cdot\|_{\cV}$, 
\begin{align*}
&a(\cdot,\cdot) \text{ is continuous if there exists } \Ccont>0\text{ such that }
|a(u,v)| \leq \Ccont \N{u}_\cV \, \N{v}_\cV \text{ for all } u, v \in \cV, 
\\
&\text{and }a(\cdot,\cdot)\text{ is {coercive} if there exists }\Ccoer\text{ such that }
|a(v,v)| \geq \Ccoer \N{v}^2_{\cV} \text{ for all } v\in \cV;
\end{align*}
``sign-definite'' is often used  as a synonym for ``coercive''. 

The relevance of continuity and coercivity to the twin goals in \S\ref{sec:accurate} and \S\ref{sec:fast} is as follows.
If a sesquilinear form $a(\cdot,\cdot)$ is both continuous and coercive, then: 
\begin{enumerate}
\item C\'ea's lemma implies that the Galerkin method for variational problems involving $a(\cdot,\cdot)$ is quasi-optimal, i.e.~for any finite dimensional subspace $\cV_N \subset \cV$, the Galerkin solution $u_N$ exists, is unique, and satisfies
\begin{equation}\label{eq:Cea}
\N{u-u_N}_{\cV} \leq \frac{C_c}{\alpha}\min_{v_N\in\cV_N}\N{u-v_N}_{\cV};
\end{equation}
moreover, if $a(\cdot,\cdot)$ is self-adjoint, i.e.~$a(u,v)= \overline{a(v,u)}$, 
then 
(by, e.g., \cite[\S2.8]{BrSc:00})
\begin{equation}\label{eq:Cea2}
\N{u-u_N}_{\cV} \leq \sqrt{\frac{C_c}{\alpha}}\min_{v_N\in\cV_N}\N{u-v_N}_{\cV}. 
\end{equation}
\item There exists bounds on the number of iterations GMRES takes to solve the linear system involving the Galerkin matrix. Indeed, these bounds, summarised in \S\ref{sec:GMRES} below, need a bound on (i) the norm of the Galerkin matrix---this follows from continuity, and (ii) the distance of the field of values (see Definition \ref{def:fov}) from the origin---this follows from coercivity.
\end{enumerate}

\begin{lemma}[Continuity and lack of coercivity of standard formulation]\label{lem:contcoer_ST}

\

\noi (i) $a_{ST}(\cdot,\cdot)$ is continuous in $\HokO$ with norm \eqref{eq:H1k} with $\Ccont \sim 1$.

\noi (ii) There exists a $k_0>0$ such that if $k\leq k_0$, then $a_{ST}(\cdot,\cdot)$ is coercive  in $\HokO$ with norm \eqref{eq:H1k} with $\Ccoer =1/2$.

\noi (iii) Let $\la_1>0$ be the first Dirichlet eigenvalue of the negative Laplacian in $\Omega$. If $k^2\geq \la_1$ then there exists a $v\in\HokO$ with $a_{ST}(v,v) = 0$.
\end{lemma}

\begin{proof}[References for proof]
(i) follows from the Cauchy--Schwarz and multiplicative trace inequalities; see, e.g., \cite[\S6.2]{Sp:15}.
For (ii), see, e.g., \cite[Lemma 6.4]{Sp:15}. For (iii), see, e.g., \cite[Lemma 6.5]{Sp:15}.
\end{proof}

\subsection{Definition of the coercive formulations of the Helmholtz IIP}\label{sec:coercive}

Although the standard variational formulation of the Helmholtz IIP (with sesquilinear form $a_{ST}(\cdot,\cdot)$) is not coercive for $k$ sufficiently large, there do exist coercive formulations of the Helmholtz IIP. These are summarised in \cite[\S I.2]{MoSp:14} (see also \cite[\S8.3]{Sp:15}). 
For these formulations discussed in \cite[\S I.2]{MoSp:14} at least one of the following is true: (i) the formulation is an integral equation on $\Gamma$; (ii) the formulation requires restricting the Hilbert space to include only piecewise solutions of the homogeneous Helmholtz equation (so-called \emph{operator-adapted} or \emph{Trefftz} spaces); (iii) the formulation is a least-squares formulation (under which any well-posed linear BVP is coercive). 

\paragraph{The MS formulation}
The paper \cite{MoSp:14} introduced a 
coercive formulation of the Helmholtz IIP, which we refer to as the MS formulation. The novelty of the MS formulation is that it is a formulation in $\Omega$ (not on $\Gamma$), does not require operator-adapted spaces, and is not a least-squares formulation.

To define the MS formulation, we first let $\nT $ denote the surface gradient on $\po$; recall that $\nT$ is such that if $v$ is differentiable in a neighbourhood of $\po$ then 
$\nT v = \gv -  \bn \dvdn $
on $\po$, where $\bn=\bn(\bx)$ is the outward-pointing unit normal vector at the point $\bx\in\po$).
Let 
\begin{equation}\label{eq:V}
V:=\left\{v: \,v\in H^1(\Omega), \; \Delta v \in L^2(\Omega),\; 
v \in H^1(\po), \,\dnv \in L^2(\po)
\right\};
\end{equation}
standard regularity results imply that the solution of the IIP is in $V$ (see \cite[Proposition 3.2]{MoSp:14}). 
In fact, the harmonic analysis results of Dahlberg, Jerison, and Kenig imply that $V=H^{3/2}(\Delta;\Omega):=\{w\in H^{3/2}(\Omega):\,\Delta w\in L^2(\Omega)\}$; see \cite[Lemme 2]{CoDa:98}.
An important feature of $V$ is that conforming FEMs in this space require $C^1$ elements \cite[Lemma 5.1]{MoSp:14}.

\begin{definition}[MS formulation]\label{def:MS}
Given $f\in \LtO, g\in \LtG$, and $k>0$, find $u\in V$ such that
\begin{equation}\label{eq:coercive}
b(u,v)= G(v) \quad\tfa v\in V,
\end{equation}
\begin{align}\label{eq:B}
\text{where}\quad 
b(u,v) &:= \int_\Omega \bigg(  \gu \cdot \bgv + k^2 u \vb + \left( \cM u + \frac{A}{k^2}\cL u\right) \overline{\cL v} \bigg) \rd \bx \\
&\hspace{-2mm}
- \int_{\GammaN} \bigg( \ri k u\,\overline{\cM v}  + \left( \bx \cdot \nT u - \ri k \beta u + \frac{d-1}{2} u \right) \overline{\dnv} + (\bx \cdot \bn)\left( k^2 u \vb - \nT u \cdot \overline{ \nT v}\right) \bigg)\rd s,\nonumber
\\
\label{eq:F}
\text{and}\quad G(v) &:= \int_\Omega \left( \overline{\cM v} - \frac{A}{k^2}\overline{\cL v}\right) f \, \rd \bx + \int_{\GammaN} \overline{\cM v} \,g \, \rd s,
\end{align}
where $\beta$ and $A$ are arbitrary real constants, $d$ is the spatial dimension, and
$$
\cM u := \bx \cdot \gu -\ri k \beta u + \frac{d-1}{2} u.
$$
\end{definition}

\begin{remark}[The MS formulation in 1-d]
In one space dimension $\Omega=(x_{-1},x_1)\subset\R$, the tangential gradient $\nabla_\Gamma$ terms drops and formulation \eqref{eq:coercive} reads
\begin{align*}
&\int_{x_{-1}}^{x_1} \bigg(
u'\overline{v}'+k^2u\overline{v}+\big((x-x_0)u'-\ri k\beta u+Ak^{-2}u''+Au\big)(\overline{v}''+k^2\overline{v})\bigg)\rd x\\
&-\!\sum_{\xi\in\{-1,1\}}\Big(\ri k (x_\xi-x_0)u(x_\xi)\overline{v}'(x_\xi)-k^2\beta u(x_\xi)\overline{v}(x_\xi)
-\ri k\beta u(x_\xi) \overline{v}'(x_\xi)\xi+(x_\xi-x_0)\xi k^2u(x_\xi)\overline{v}(x_\xi)\Big)\\
&=
\int_{x_{-1}}^{x_1} \Big((x-x_0)f\overline{v}'+(\ri k\beta-A)f\overline{v}-Ak^{-2}f\overline{v}''\Big)\rd x
+\sum_{\xi\in\{-1,1\}}\Big((x_\xi-x_0) \overline v'(x_\xi)+\ri k \beta \overline v(x_\xi)\Big)g(x_\xi)
\end{align*}
for all $v\in V$ where $x_0\in(x_{-1},x_1)$. 
Here the coefficient $\xi=\pm1$ in the boundary terms gives the correct sign to the flux terms.
\end{remark}

The MS formulation  comes from integrating over $\Omega$ the identity 
\begin{align}\label{eq:Morawetz_intro}
\overline{\cM v} \cL u + \cM u \bLv = \dive \Big[ \overline{\cM v}\, \gu + \cM u\, \bgv + \bx( k^2 u\vb-  \gu \cdot \bgv) \Big] 
-\gu \cdot \bgv -k^2 u \vb,
\end{align}
using the PDE \eqref{eq:bvp_main1} and the boundary conditions \eqref{eq:bvp_main2} and then adding on the least-squares-type term $\cL u\,\overline{\cL v}$. Multipliers of the form $\cM v$ were first used for the Helmholtz equation by Morawetz and Ludwig in \cite{MoLu:68} and Morawetz in \cite{Mo:75}; see the discussion in \cite[\S I.4]{MoSp:14} and \cite[Remark 2.7]{SpKaSm:15}.

A generalisation of this formulation to the IIP with the PDE \eqref{eq:bvp_main1} replaced by $\Delta u +k^2 nu =-f$, and with $n$ satisfying conditions that guarantee nontrapping of rays (see \cite[\S6]{GrPeSp:18}) was introduced by Ganesh and Morgenstern \cite{GaMo:17b}; this formulation arises by integrating over $\Omega$ the analogue of the identity \eqref{eq:Morawetz_intro} with $\cL$ replaced by $\cL_n:= \Delta + k^2 n$.

\paragraph{Least-squares formulation}

The least-squares formulation of the IIP is posed in the same space as the MS formulation, i.e.~$V$, and it is therefore natural to compare the two.

\begin{definition}[Least-squares variational formulation in $V$]\label{def:LS}
Given $f\!\in\!\LtO$, $g\in\!\LtG$, and $k>0$, find $u\in V$ such that
\begin{equation}\label{eq:vfLS}
a_{LS}(u,v)= F_{LS}(v) \quad\tfa v\in V,
\end{equation}
where 
$$
a_{LS}(u,v):=\int_\Omega \cL u \, \overline{\cL v} \,\rd \bx  + \int_\Gamma \big(\partial_n u - \ri  k\,u\big)
 \overline{\big(\partial_n v -\ri k  v\big)}\rd s
\quad\text{and}\quad
F_{LS}(v):= \int_\Omega f \,\vb \,\rd \bx + \int_\Gamma g\, \vb\, \rd s.
$$
\end{definition}

\begin{remark}[The MS formulation as a  ``stabilised method''] \quad
Formulation \eqref{eq:coercive} is a special case of a slightly more general family \cite[(3.4)]{MoSp:14}:
\begin{align*}
b_{Z}(u,v) = &\!\int_\Omega\! \bigg( \big(2-d +Z_1+Z_2\big)\gu \cdot \bgv 
+\big( d-Z_1-Z_2\big)k^2 u \overline v 
+ \left( \bx \cdot \nabla u + Z_2 u + \frac{A}{k^2}\cL u\right) \overline{\cL v} \bigg) \rd \bx 
\\
& - \int_{\GammaN} \bigg( \ri k u \,(\bx \cdot \overline{\nabla v} + Z_1\overline v)+ \left( \bx \cdot \nT u  + Z_2 u \right) \overline{\dvdn} + (\bx \cdot \bn)\left( k^2 u \overline v - \nT u \cdot \overline{ \nT v}\right) \bigg)\rd s,
\\
G_{Z}(v) = &\int_\Omega \left( \bx \cdot \overline{\nabla v} + Z_1\overline v - \frac{A}{k^2}\overline{\cL v}\right) f \, \rd \bx + \int_{\GammaN} (\bx \cdot \overline{\nabla v} + Z_1\overline v) \,g \, \rd s,
\end{align*}
where $Z_1,Z_2$ are complex parameters 
which, in \cite{MoSp:14}, were written as 
$Z_1=\alpha_1+\ri k\beta_1,Z_2=\alpha_2-\ri k \beta_2$.
Formulation \eqref{eq:coercive} corresponds to the choice $Z_1=\overline{Z_2}=\frac{d-1}2+\ri k \beta$.
This formulation is consistent and continuous in $V$ for any choice of $Z_1,Z_2\in\C$, while coercivity is ensured by a certain range of parameters only (\cite[Theorem~3.4]{MoSp:14}).
We can decompose it into the sum of  four terms:
\begin{align}\label{eq:AFdecomposition}
b_{Z}(u,v) = &  Z_1 a_{ST}(u,v) + \frac A{k^2} \int_\Omega\cL u\overline{\cL v}\,\rd\bx + Z_2 a_{0} (u,v) + a_{\bx}(u,v),\\
G_{Z}(v) = & Z_1 F_{ST}(v) + \frac A{k^2} \int_\Omega (-f)\overline{\cL v}\,\rd\bx +  F_{\bx}(v),
\nonumber
\end{align}
where $a_{ST}$ and $F_{ST}$ are from Definition~\ref{def:standard}, 
\begin{align*}
a_{0} (u,v) := \int_\Omega (\gu \cdot \bgv -k^2 u \overline v+u \overline{\cL v})\rd \bx  - \int_{\GammaN} u\overline{\dvdn}\rd s 
=0 \qquad\tfa u,v\in V
\end{align*}
from integration by parts (which means that the choice of $Z_2$ is irrelevant) and
\begin{align*}
a_{\bx}(u,v) := &\int_\Omega \bigg( (2-d)\gu \cdot \bgv + dk^2 u \overline v + ( \bx \cdot \nabla u ) \overline{\cL v} \bigg) \rd \bx 
\\
& - \int_{\GammaN} \bigg( \ri k u \,\overline{\bx \cdot \nabla v}+  \bx \cdot \nT u  \overline{\dvdn} + (\bx \cdot \bn)\left( k^2 u \overline v - \nT u \cdot \overline{ \nT v}\right) \bigg)\rd s,
\\
F_{\bx}(v) := &\int_\Omega  \overline{\bx \cdot \nabla v} \, f \, \rd \bx + \int_{\GammaN} \overline{\bx \cdot \nabla v} \,g \, \rd s.
\end{align*}
We have that $a_{\bx}(u,v)=F_{\bx}(v)$ for all $v\in V$ from combining (i) the expressions of $f$ and $g$, (ii) the divergence theorem applied to the Rellich identity \cite[eq.~(1.32)]{MoSp:14} and (iii) the divergence theorem applied to ($k^2$ times) 
$\dive[u\overline v\bx]=(\bx\cdot\nabla u)\overline v+(\bx\cdot\nabla \overline v)u+du\overline v$.

The decomposition \eqref{eq:AFdecomposition} shows that the MS formulation can be seen as a ``stabilised method'' \cite{Bu:13}, related to the ``Galerkin-least squares'' (GLS) method \cite{ThPi:95,HaHu:92}:
$b_{Z}(u,v) =G_Z(v)$ is a linear combination of the standard formulation \eqref{eq:vfH1}, the volume part of the least-squares formulation, and a consistent formulation $a_{\bx}(u,v) =F_\bx(v)$ arising from Rellich's identities. 
\end{remark}

\subsection{Continuity and coercivity of least-squares and MS formulations}

The continuity and coercivity properties of the least-squares and MS formulations depend on what norm is used for the space $V$ \eqref{eq:V}.

\begin{definition}[The norms $\N{\cdot}_{V_1}$, $\N{\cdot}_{V_2}$]
Let
\begin{align}
\N{v}_{V_1}^2:=&k^{-2}\N{\Delta v}^2_{\LtO}+\N{\nabla v}^2_{\LtO}+k^2\N{v}^2_{\LtO}
+\length \left( \N{\dnv}_{\LtG}^2+\N{\nT v}_{\LtG}^2 +k^2\N{v}^2_{\LtG}
\right), \label{eq:normV}
\\
\label{eq:Valt}
\N{v}_{V_2}^2:=&
\N{\cL v}^2_{\LtO}+\N{\nabla v}^2_{\LtO}+k^2\N{v}^2_{\LtO} +L\left(\N{\dnv}_{\LtG}^2+ \N{\nT v}^2_{\LtG}+k^2\N{v}^2_{\LtG}\right),
\end{align}
where $\length$ is the diameter (or some other characteristic length scale) of the domain.
\end{definition}
\noi Two remarks:
\begin{enumerate}
\item
We weight the derivatives by $k$ and include $\length$ in front of the boundary terms so that, when computed for plane-wave solutions of the homogeneous Helmholtz equation with wavenumber $k$, each term of the norm scales in the same way as $k$ and $\length$ vary; see \cite[Remark 3.8]{MoSp:14}. 
The norm equivalence
$
\frac{1}{\max\{3,2k^{-2}\}}
 \N{v}^2_{V_1} \leq \N{v}^2_{V_2} \leq (2k^2+1)
\N{v}^2_{V_1}
$
holds, 
and if $\cL v=0$ then
$
\half\N{v}^2_{V_1} \leq \N{v}^2_{V_2} \leq 
\N{v}^2_{V_1}.
$
\item The sharp bound \eqref{eq:bound} below shows that, for the solution $u$ of the IIP, each term in $\N{u}^2_{V_2}$ is of the same order.
\end{enumerate}

\begin{definition}[Star-shaped with respect to a ball]

\

(i) $\Omega$ is \emph{star-shaped with respect to $\bx_0\in \Omega$} if, whenever $\bx \in \Omega$, the segment $[\bx_0,\bx]\subset \Omega$.

(iii) $\Omega$ is \emph{star-shaped with respect to the ball $B_{a}(\bx_0)$} if it is star-shaped with respect to every point in $B_{a}(\bx_0)$.

(iii) $\Omega$ is \emph{star-shaped with respect to a ball} if there exists $a>0$ and $\bx_0\in\Omega$ such that $\Omega$ is star-shaped with respect to the ball $B_{a}(\bx_0)$.
\end{definition}

Recall that if $\Omega$ is Lipschitz, then it is star-shaped with respect to $B_{a}(\bx_0)$ if and only if
$(\bx-\bx_0) \cdot \bn(\bx) \geq {a}$ for all  $\bx \in \Gamma$ for which $\bn(\bx)$ is defined; see, e.g., \cite[Lemma 5.4.1]{Mo:11}.

\begin{theorem}[$k$-explicit bound on the solution of the IIP]\label{thm:bound}
If $\Omega$ is \emph{either} Lipschitz and star-shaped \emph{or} smooth (i.e.~$C^\infty$), then given $k_0>0$, the solution of the IIP satisfies
\begin{equation}\label{eq:bound}
\N{u}_{V_1}  \lesssim L\N{f}_\LtO + L^{1/2}\N{g}_\LtG\tfa k\geq k_0;
\end{equation}
moreover, this bound is sharp in its $k$-dependence.
\end{theorem}

\begin{proof}[References for the proof of Theorem \ref{thm:bound}]
The bound is proved for Lipschitz star-shaped $\Omega$ in \cite[Remark 3.6]{MoSp:14} and for general smooth $\Omega$ in \cite[Theorem 1.8, Corollary 1.9]{BaSpWu:16}. Note that \cite[Corollary 1.9]{BaSpWu:16} does not include the $\|\dnu\|_\LtG$ or $\|\Delta u\|_{L^2(\Omega)}$ terms; the former can be included by an argument involving Green's identity essentially identical to the proof of this corollary (this argument in the case $f=0$ is in, e.g., \cite[Lemma 4.2]{Sp:14}). The latter can be included in a straightforward way using the PDE \eqref{eq:bvp_main1}.
The sharpness with respect to $k$ is proved in \cite[Lemma 4.10]{Sp:14} and  \cite[Lemma 5.5]{BaSpWu:16}.
Note that the bound $\|u\|_{\HokO} \lesssim L\N{f}_\LtO + L^{1/2}\N{g}_\LtG$ (contained in \eqref{eq:bound}) was proved when $\Omega$ is star-shaped with respect to a ball with smooth boundary in \cite[Proposition 8.1.4]{Me:95} for $d=2$ and \cite[Theorem 1]{CuFe:06} for $d\geq 3$.
\end{proof}

In the rest of the paper, we allow the constants in $\lesssim$ and $\sim$ to depend on $L$.

\begin{remark}
Proving that the sharp bound \eqref{eq:bound} holds when $\Omega$ is a general Lipschitz domain is still open. The best results in this direction are in \cite[Theorem 1.6]{Sp:14}, and show that, when $L=1$, 
$
\N{u}_{V_1}\lesssim k \N{f}_\LtO + k^{1/2} \N{g}_\LtG
$
in the general Lipschitz case, and 
$
\N{u}_{V_1}\lesssim  k^{3/4} \N{f}_\LtO + k^{1/4}\N{g}_\LtG
$
when $\Omega$ is piecewise-smooth. 
\end{remark}

\begin{lemma}[Continuity and coercivity of least-squares formulation]\label{lem:LSCC}\

\noi (i) In the norm $\N{\cdot}_{V_1}$, $a_{LS}(\cdot,\cdot)$ is continuous with $\Ccont\sim k^2$. 
If $\Omega$ is \emph{either} Lipschitz and star-shaped with respect to a ball, \emph{or} $C^\infty$, then $a_{LS}(\cdot,\cdot)$ is coercive in the norm $\N{\cdot}_{V_1}$
with $\Ccoer \sim 1$.

\noi (ii) In the norm $\N{\cdot}_{V_2}$, $a_{LS}(\cdot, \cdot)$ is continuous with $\Ccont \sim 1$. 
If $\Omega$ is \emph{either} Lipschitz and star-shaped with respect to a ball, \emph{or} $C^\infty$, then 
$a_{LS}(\cdot,\cdot)$ is coercive in the norm $\N{\cdot}_{V_2}$ with $C_{\coer}\sim 1$.
\end{lemma}

\begin{proof}[Proof of Lemma \ref{lem:LSCC}]
The continuity results follow from the Cauchy--Schwarz and triangle inequalities.
The coercivity results follow from  Theorem \ref{thm:bound}.
\end{proof}

\begin{lemma}[Continuity of the MS formulation]\label{lem:MScont}
If 
\begin{equation}\label{eq:beta2}
\beta \leq CL
\end{equation}
 for some $C>0$ independent of $k$ and $L$, then 

\noi (i) $b(\cdot,\cdot)$ is continuous in the norm $\N{\cdot}_{V_1}$ with $\Ccont \sim k$.

\noi (ii) $b(\cdot,\cdot)$ is continuous in the norm $\N{\cdot}_{V_2}$ with $\Ccont \sim 1$.

\end{lemma}
\begin{proof}
(i) is proved in \cite[Lemma 3.3]{MoSp:14} using the Cauchy-Schwarz inequality; the proof of (ii) follows in an almost identical way.
\end{proof}

\begin{theorem}[Coercivity of MS formulation]\label{thm:MScoer}
Let $\Omega$ be a Lipschitz domain with diameter $\length$ that is star-shaped with respect to a ball (without loss of generality at the origin), i.e.~there exists a $\gamma>0$ such that
$
\bx \cdot \bn(\bx) \geq \gamma \length
$
for all $\bx \in \po$ such that $\bn(\bx)$ exists.
Assume that 
\begin{equation}\label{eq:beta1}
\beta \geq \frac{ \length}{2}\left(1+ \frac{4}{\gamma} + \frac{\gamma}{2}\right).
\end{equation}

\noi (i) If $A=1/3$, then, for any $k>0$, \;
$\Re b(v,v) \geq \frac{\gamma}{4} \N{v}^2_{V_1}$ \; for all $v \in V$.

\noi (ii) If $A=k^2$, then for any $k>0$, \;
$\Re b(v,v) \geq \frac{\gamma}{4} \N{v}^2_{V_2}$ \; for all $v \in V$.
\end{theorem}

\begin{corollary}\label{cor:MS}
For the MS formulation (with sesquilinear form $b(\cdot,\cdot)$), if $\beta$ satisfies both \eqref{eq:beta2} and \eqref{eq:beta1}, we have $\Ccont/\Ccoer\sim  k$ in the $\N{\cdot}_{V_1}$ norm (if $A=1/3$), and $\sim 1$ in the $\N{\cdot}_{V_2}$ norm (if $A=k^2$).
\end{corollary}
\begin{proof}[Proof of Theorem \ref{thm:MScoer}]
(i) is proved in \cite[Theorem 3.4]{MoSp:14} using the identity \eqref{eq:Morawetz_intro}.
(ii) follows from the proof of \cite[Theorem 3.4]{MoSp:14} taking $A=k^2$ in \cite[Equation~(3.9)]{MoSp:14}, and then dealing with the terms on $\Gamma$ exactly as before. (Note that \cite[Equation~(3.11)]{MoSp:14} has a typo: $\cL v$ should be $\Delta v$.)
\end{proof}

\section{Accuracy of Galerkin solutions}\label{sec:test_acc}

\subsection{The discrete space}\label{sec:Vn}

To compare the properties of the different variational formulations (\eqref{eq:vfH1}, \eqref{eq:coercive} and \eqref{eq:vfLS}), we apply the Galerkin method to each formulation with the same discrete space $V_N\subset V$.
As explained in \cite[Lemma~5.1]{MoSp:14}, the elements of $V_N$ must be in $C^1(\overline\Omega)$.
For simplicity we restrict ourselves to the cases $d=1$ and $d=2$.

In one space dimension, we simply choose $V_N$ to be the Hermite element space, i.e.\ the space of $C^1$ piecewise-cubic polynomials: 
if $\Omega=(x_0,x_{n})\subset\R$, given a mesh with nodes $x_0<x_1<\cdots<x_{n-1}<x_n$, 
$$
V_N = \big\{v\in C^1(x_0,x_n),\; v|_{(x_{j-1},x_{j})}\text{ is a polynomial of degree }\le3,\; j=1,\ldots,n\big\}.
$$
The meshwidth is defined as $h=\max_{j=1,\ldots,n}(x_{j}-x_{j-1})$ and $N:=\dim V_N=2n+2$.
The degrees of freedom are function values and first derivatives in the nodes $x_j$.

In two space dimensions we consider a rectangular domain $\Omega=(x_0,x_{n_x})\times(y_0,y_{n_y})\subset\R^2$, 
for $x_0<x_1<\cdots<x_{n_x-1}<x_{n_x}$ and $y_0<y_1<\cdots<y_{n_y-1}<y_{n_y}$,  
$$
V_N =  \big\{v\in C^1\big((x_0,x_{n_x})\times(y_0,y_{n_y})\big),\; 
v|_{(x_{j-1},x_{j})\times(y_{j'-1},y_{j'})}\in \mathbb{Q}^3
,\; j=1,\ldots,n_x,\;  j'=1,\ldots,n_y\big\}.
$$
Here $\mathbb Q^3$ is the space of polynomials of degree at most 3 separately in the $x$ and $y$ variables.
The meshwidth is $h=\max_{j=1,\ldots,n_x; j'=1,\ldots,n_y}((x_{j}-x_{j-1})^2+(y_{j'}-y_{j'-1})^2)^{1/2}$ and $N:=\dim V_N=(2n_x+2)(2n_y+2)$.

In the numerical experiments we will consider only uniform meshes, i.e.\ with identical elements.
The theory presented below can be easily adapted to different spaces with higher polynomial degrees and/or continuity constraints and defined on suitable curvilinear domains
(one would then modify the proof of Lemma~\ref{lem:H1error} using the general results of \cite{BeBuRiSa:11}).

\subsection{Error bounds from continuity and coercivity}
We denote by $u_N$ the solution of the Galerkin method in $V_N$ applied to one of the formulations \eqref{eq:vfH1}, \eqref{eq:coercive} or \eqref{eq:vfLS}; the choice of the formulation will be clear from the context.

\begin{lemma}[Quasi-optimality of Galerkin method]\label{lem:QuasiOpt}\quad
Let $V_N$ be a finite-dimensional subspace of $V$.
(i) If the MS formulation with $A=1/3$ is solved using the Galerkin method, then, given $k_0>0$, 
\begin{equation}\label{eq:qo1}
\N{u-u_N}_{\HokO}\le
\N{u-u_N}_{V_1} \lesssim k \min_{v_N\in V_N}\N{u-v_N}_{V_1}\quad\tfa k\geq k_0.
\end{equation}
(ii)
If \emph{either} the least-squares formulation \emph{or} the MS formulation with $A=k^2$ are solved using the Galerkin method, then, given $k_0>0$, 
\begin{equation}\label{eq:qo2}
\N{u-u_N}_{\HokO}\le
\N{u-u_N}_{V_2} \lesssim \min_{v_N\in V_N}\N{u-v_N}_{V_2}\quad\tfa k\geq k_0.
\end{equation}
\end{lemma}

\begin{proof}
For the MS formulation, the bounds follow from C\'ea's lemma (for non-self-adjoint sesquilinear forms) \eqref{eq:Cea} and Corollary \ref{cor:MS}. 
For the LS formulation, the bounds follow from C\'ea's lemma (for self-adjoint sesquilinear forms) \eqref{eq:Cea2} and
Lemma \ref{lem:LSCC}(ii).
\end{proof}

\begin{remark}\label{rem:WNormSucks}
The different $k$-dependences of the quasi-optimality constants in Lemma \ref{lem:QuasiOpt} suggests we should prefer case (ii), i.e.~we should 
use the MS formulation with $A=k^2$ or the least squares formulation and work in the $\N{\cdot}_{V_2}$ norm.
However, we see from numerical experiments that the quasi-optimality in  $\N{\cdot}_{V_2}$ norm does not provide a good accuracy in $H^1_k(\Omega)$ and $L^2(\Omega)$ norms (despite these being bounded by $\N{\cdot}_{V_2}$).
This is because the best-approximation error in  $\N{\cdot}_{V_2}$ in the regimes of interest can be much larger than the best-approximation error in the other norms considered.
The numerical experiments in \S\ref{sec:acc_num} show that the difference between the best-approximation errors in the different norms can outweigh the different $k$-dependence of the quasi-optimality constants.

Figure \ref{fig:uPlot} shows a simple and representative example in one dimension.
The blue dashed lines represent the real part of $u(x)=\re^{\ri k x}$, solution of \eqref{eq:bvp_main} in $\Omega=(0,1)$ with $f=0$, $g(0)=-2\ri k$, $g(1)=0$ and $k=30\pi\approx 94.25$.
The red curves depict the real parts of the projections on $V_N$, with $n=100$ elements (corresponding to 13.47 degrees of freedom per wavelength), orthogonal with respect to the $\N{\cdot}_{V_1}$ norm (left panel) and with respect to the $\N{\cdot}_{V_2}$ norm (right panel).
In both cases we chose $L=1$.
The $\N{\cdot}_{V_1}$-orthogonal projection is visually indistinguishable from the exact solution, while the amplitude of the $\N{\cdot}_{V_2}$-orthogonal projection is less than half the correct one.

Table \ref{tab:uPlot} shows the relative error measured in four norms ($\N{\cdot}_{L^2(\Omega)}$, $\N{\cdot}_{H^1_k(\Omega)}$, $\N{\cdot}_{V_1}$, $\N{\cdot}_{V_2}$) of the four orthogonal projections corresponding to the same four norms.
The values in the table confirm quantitatively what is visible in the plot:
the $\N{\cdot}_{V_1}$-orthogonal projection has $L^2(\Omega)$ and $H^1_k(\Omega)$ relative errors comparable to the best approximation in these norms, while the $\N{\cdot}_{V_2}$-orthogonal projection has much larger relative errors.
We observe similar phenomena also for non-homogeneous problems ($f\ne0$).
\end{remark}
\begin{figure}[htb]
\centering
\includegraphics[width=74mm, clip, trim= 50 0 50 0]{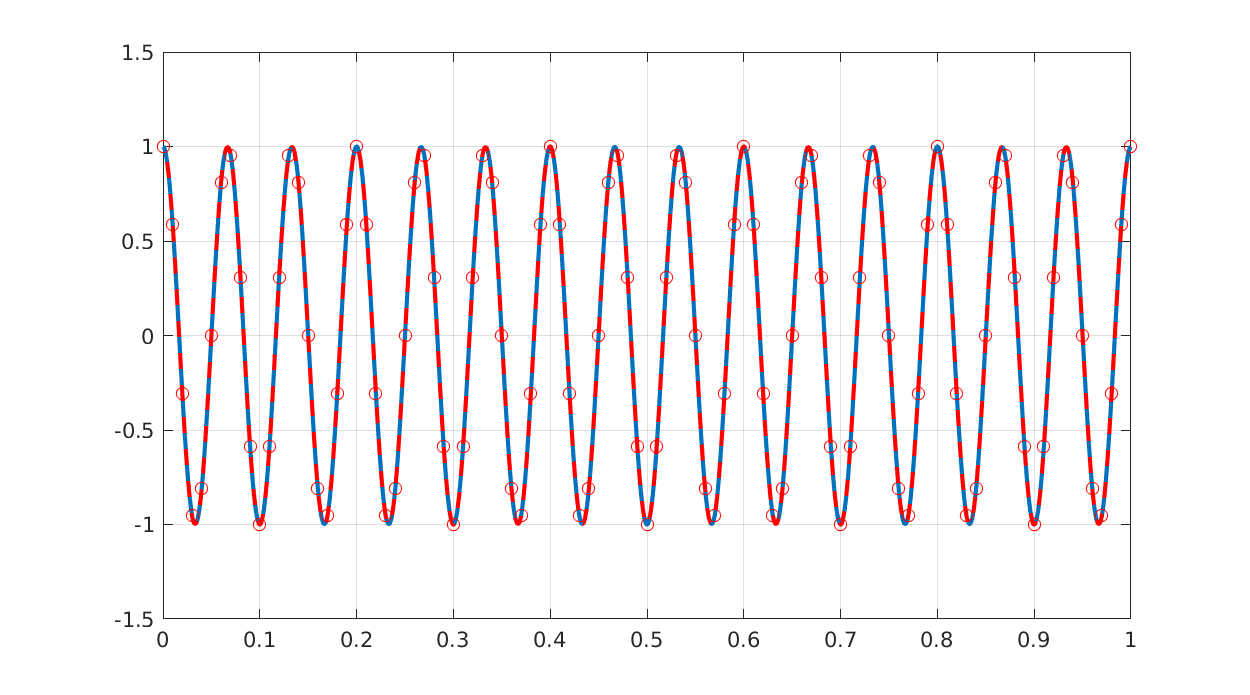}
\includegraphics[width=74mm, clip, trim= 50 0 50 0]{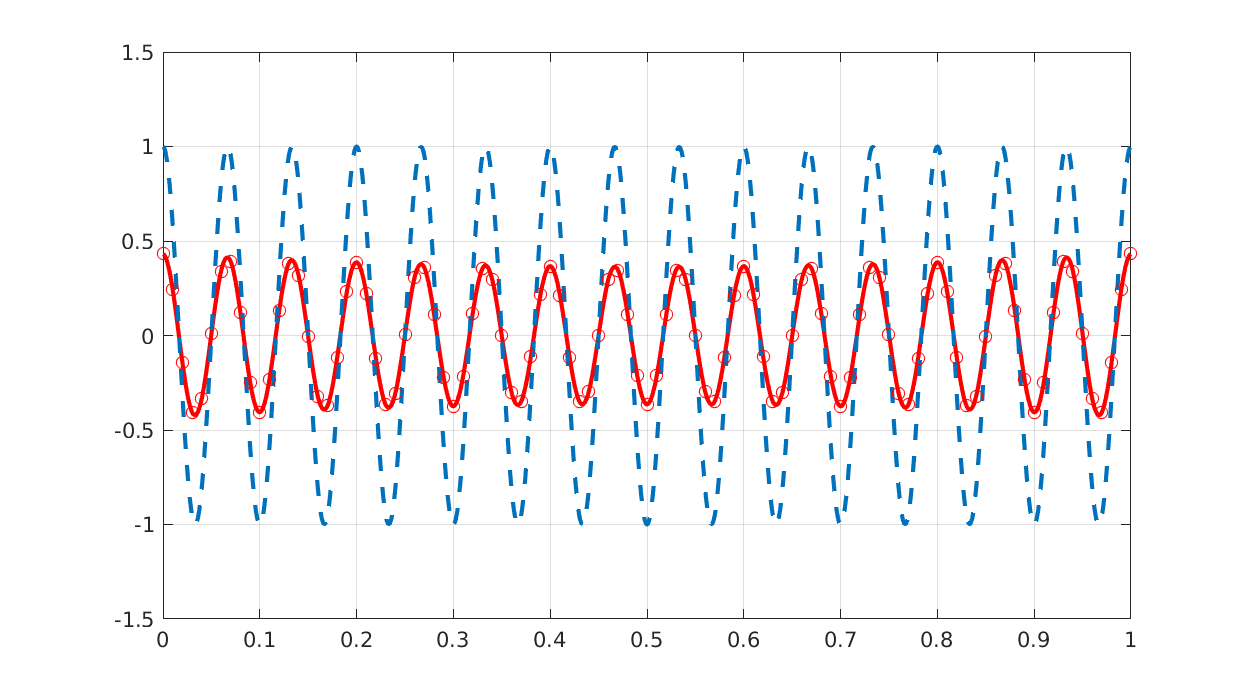}
\caption{Blue dashed line: the real part of $u(x)=\re^{\ri 30\pi x}$ in $\Omega=(0,1)$;
red line: the projection of $u$ on $V_N$ for $n=100$ orthogonal with respect to the $\N{\cdot}_{V_1}$ norm (left panel) and with respect to the $\N{\cdot}_{V_2}$ norm (right panel).
See Remark~\ref{rem:WNormSucks} and Table~\ref{tab:uPlot}.}
\label{fig:uPlot}
\end{figure}
\begin{table}[htb]
\begin{tabular}{|l |l l l l|}\hline
& $\N{\cdot}_{L^2(\Omega)}$ rel.\ err. &$\N{\cdot}_{H^1_k(\Omega)}$ rel.\ err. &$\N{\cdot}_{V_1}$ rel.\ err. &$\N{\cdot}_{V_2}$ rel.\ err.\\\hline
$\N{\cdot}_{L^2(\Omega)}$-orthog.\ proj. 
&  0.000556   &    0.00308   &      0.0173  &   1.46 \\
$\N{\cdot}_{H^1_k(\Omega)}$-orthog.\ proj. 
& 0.000574   &    0.00301  &      0.0143 &    1.35\\
$\N{\cdot}_{V_1}$-orthog.\ proj. 
& 0.000824  &    0.00333   &    0.0125  &   1.23\\
$\N{\cdot}_{V_2}$-orthog.\ proj. 
& 0.615     &    0.615  &      0.589 &   0.764
\\\hline
\end{tabular}
\caption{The relative errors measured in four different norms of the best approximations in the same four norms of $u(x)=\re^{\ri 30\pi x}$ in $\Omega=(0,1)$, with $V_N$ for $n=100$.
See Remark~\ref{rem:WNormSucks} and Figure~\ref{fig:uPlot}.
}
\label{tab:uPlot}
\end{table}

\begin{remark}[Sharpness of the quasi-optimality constants]\label{rem:QuasiOptSurface}
The linear dependence on $k$ in the quasi-optimality bound in \eqref{eq:qo1} comes from the ratio between coercivity and continuity constants. 
We believe this $k$-dependence is not sharp.
Figure \ref{fig:QuasiOptSurface} shows (in logarithmic scale) the empirical quasi-optimality ratio $C_{qo}(u,k,V_N):=\N{u-u_N}_{V_1}/\min_{v_N\in V_N}\N{u-v_N}_{V_1}$ between Galerkin and best-approximation error for a one-dimensional experiment in $\Omega=(0,1)$ with solution $u(x)=\re^{\ri k x}$ for several values of $h$ and $k$.
As expected, the quasi-optimality ratio $C_{qo}(u,k,V_N)$ is close to 1 for large values of $kh$ (when $V_N$ does not contain any good approximation of $u$) and for small values of $hk$ ($V_N$ is sufficiently fine to overcome the pollution effect).
(Actually for $hk\lesssim 10^{-2}$ the ratio appears to be smaller than 1 because the $V$-orthogonal projection used to compute the best-approximation error is ill-conditioned, so the values displayed for this regime are not reliable; this is visible in the left corner of the figure.)
The most interesting region is the ``ridge'' crossing diagonally the $\log_{10}h$/$\log_{10}k$ plane, corresponding to the pollution regime.
The maximal value of $C_{qo}(u,k,V_N)$ over the considered spaces $V_N$ (i.e.\ over the values of $h$) for each $k$ is represented by the red continuous curve.
The black dashed line is its best linear fit, showing that the empirical quasi-optimality ratio grows like 
$$C_{qo}(u,k,V_N)\approx k^{0.407},$$
i.e.\ slower than the linear growth predicted by \eqref{eq:qo1}.
Other boundary value problems, e.g.\ non-homogeneous ones, give similar plots with exponents between $0.4$ and $0.41$.
\end{remark}
\begin{figure}[htb]
\centering
\includegraphics[width=100mm, clip, trim= 20 20 40 30]{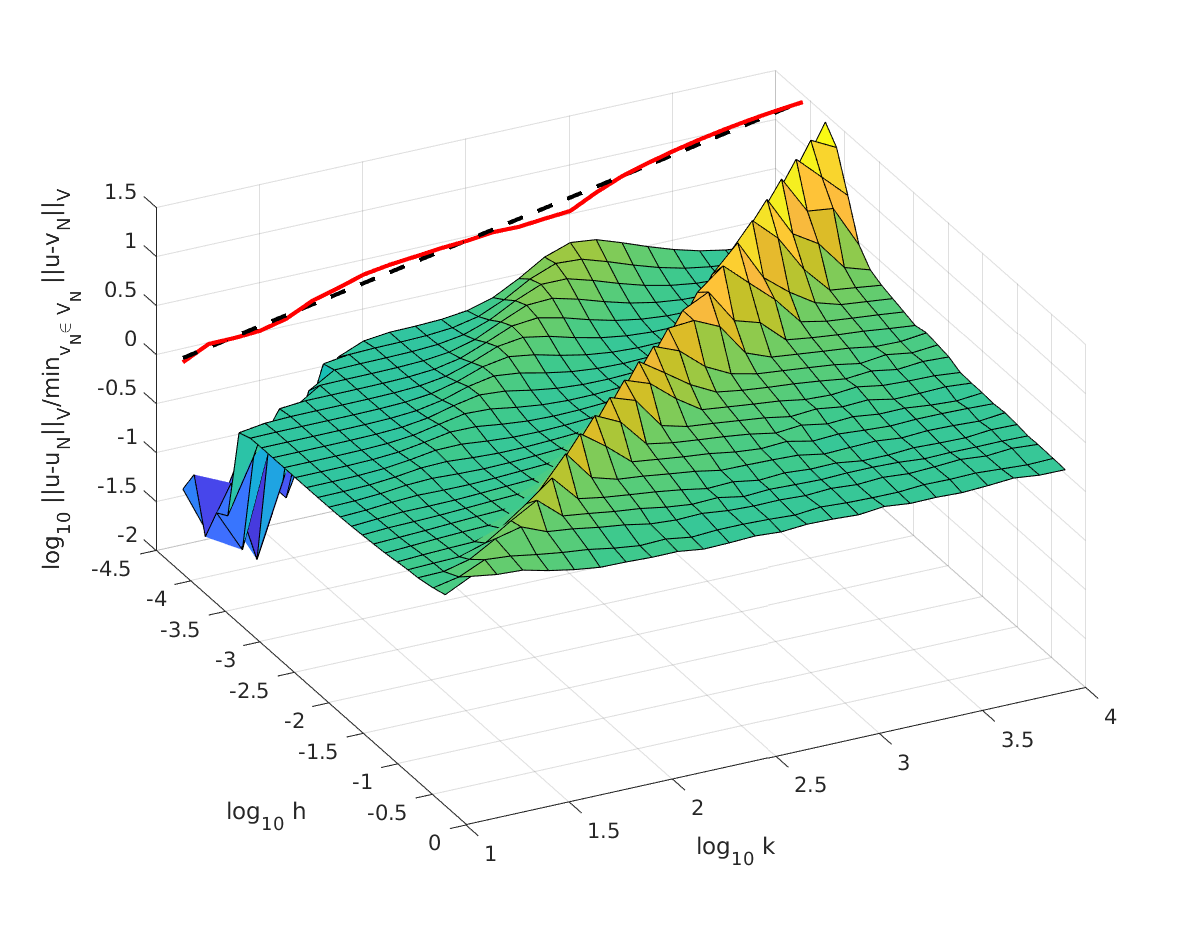}
\caption{The empirical quasi-optimality ratio $\N{u-u_N}_{V_1}/\min_{v_N\in V_N}\N{u-v_N}_{V_1}$ for $u(x)=\re^{\ri k x}$ in $\Omega=(0,1)$ as function of $h$ and $k$.
See Remark \ref{rem:QuasiOptSurface} for details.}
\label{fig:QuasiOptSurface}
\end{figure}

In order to convert the quasi-optimality bounds \eqref{eq:qo} and \eqref{eq:qo2} into bounds on the relative error, we make the following assumption.

\begin{assumption}[Oscillatory behaviour of Helmholtz solutions]\label{ass:1}
Assume that $u\in H^{4}(\Omega)$ and 
\begin{equation}\label{eq:1}
\N{u}_{H^{m+1}(\Omega)} \lesssim k\N{u}_{H^m(\Omega)}, \quad m=0,1,2,3.
\end{equation}
\end{assumption}

\begin{remark}[Discussion of Assumption \ref{ass:1}]\label{rem:ass1}
Assumption \ref{ass:1} concerns the behaviour of $u$ as a function of $k$ when $k\tendi$. When generalisations of the IIP are used to model the scattering and propagation of waves, the data $f$ and $g$ depends on $k$.
Assumption \ref{ass:1} is therefore implicitly an assumption on $f$ and $g$, and it can be violated by choosing $f$ and $g$ such that the solution oscillates on a scale smaller than $k^{-1}$, like, e.g., the plane wave solution $\exp (\ri k^2 \bx\cdot \ba)$. 
However, with physically realistic $f$ and $g$ one expects their scales of oscillation to match the $k^{-1}$ scale inherent in the Helmholtz operator $\Delta +k^2$.
For example, scattering by an incident plane wave with frequency $k^2$, $\exp (\ri k^2 \bx\cdot \ba)$, would be modelled by the Helmholtz equation $\Delta +(k^2)^2$ and not by the equation $\Delta +k^2$. 

Two situations where \eqref{eq:1} has been established are the following:
In 1-d, solutions of $\Delta u +k^2u=0$ can be written explicitly in terms of plane waves, i.e.~trigonometric functions, and then it is straightforward to prove \eqref{eq:1}; see \cite[\S3.4]{IhBa:95}.
In higher dimensions, for the sound-soft scattering problem with incident plane-wave $\exp (\ri k \bx\cdot \ba)$, the analogue of \eqref{eq:1} was proved for the Neumann trace of the solution in \cite[Theorems 1.1 and 1.2]{GrLoMeSp:15} and \cite[Theorem 1.16 and Equation 4.10]{GaMuSp:17}.
\end{remark}

\begin{lemma}\label{lem:VboundedBySobolev}
For all $w\in V\cap H^2(\Omega)$ and $k\ge k_0>0$, we have
\begin{equation}\label{eq:VboundedBySobolev}
\N{w}_{V_1}\lesssim k^{1/2} \N{w}_\HokO + k^{-1/2}\N{w}_{H^2(\Omega)},
\qquad
\N{w}_{V_2}\lesssim k \N{w}_\HokO +  \N{w}_{H^2(\Omega)}.
\end{equation}
\end{lemma}
\begin{proof}
Using the definition of $\cL$, we find
$$
\N{\cL w}_\LtO \lesssim \N{w}_{H^2(\Omega)}+ k^2 \N{w}_{L^2(\Omega)}
$$
and similarly 
$$
k^{-1}\N{\Delta w}_\LtO \lesssim k^{-1}\N{w}_{H^2(\Omega)}.
$$
The multiplicative trace inequality 
$$
\N{w}^2_\LtG \lesssim \N{w}_{\LtO}\N{w}_{\HoO}
$$
(see, e.g., 
\cite[Theorem 1.6.6]{BrSc:00}) implies that, given $k_0>0$,
\begin{equation}\label{eq:Euan1}
k\N{w}^2_\LtG\lesssim \N{w}^2_{\HokO} \quad \text{ i.e.} \quad k^{1/2}\N{w}_\LtG\lesssim \N{w}_{\HokO}, 
\end{equation}
for all $k\geq k_0$, 
and thus
$$
k\N{w}_\LtG\lesssim k^{1/2} \N{w}_{\HokO}.
$$
Now, using \eqref{eq:Euan1} applied to $\nabla w\in H^1(\Omega)$, we have,
$$
\N{\partial_n w}_\LtG \lesssim \N{\gamma(\nabla w)}_\LtG \lesssim 
k^{-1/2} \N{w}_{H^2(\Omega)}+ k^{1/2} \N{w}_{\HoO},
$$
where here we have used $\gamma$ to denote explicitly the trace operator on $\Gamma$.
Similarly
$$
\N{\nT w}_\LtG \lesssim \N{\gamma(\nabla w)}_\LtG \lesssim 
k^{-1/2} \N{w}_{H^2(\Omega)}+ k^{1/2} \N{w}_{\HoO}.
$$
Summing all terms in the definition of the norms $\N{\cdot}_{V_1}$ and $\N{\cdot}_{V_2}$ we obtain the assertion.
\end{proof}

The next lemma derives a priori error bounds by combining:
(i) the quasi-optimality properties of the formulations from Lemma~\ref{lem:QuasiOpt},
(ii) the scaling properties of the solutions from Assumption~\ref{ass:1},
(iii) the scaling properties of the norms on $V$ from Lemma~\ref{lem:VboundedBySobolev},
and (iv) the approximation properties of the discrete space from \cite{BeBuRiSa:11}.

\begin{proposition}[Bound on the relative $H^1$-error]\label{lem:H1error}
Assume that $u$ satisfies Assumption \ref{ass:1}, $d\le2$, and that the discrete space $V_h$ is as in \S\ref{sec:Vn}.

(i) If the MS formulation with $A=1/3$ is solved using the Galerkin method, then, given $k_0>0$, 
\begin{equation}\label{eq:3new1}
\N{u-u_N}_{\HokO} \leq \N{u-u_N}_{V_1} \lesssim h^2 k^{7/2} (1+k^2h^2) \N{u}_{\HokO} \qquad\tfa k\ge k_0,
\end{equation}
i.e., the goal \eqref{eq:rel_error} is achieved if $h^2k^{7/2}$ is sufficiently small, i.e.~the method is $hk^{7/4}$-accurate (in the sense of Definition \ref{def:hka}).

(ii) If \emph{either} the least-squares formulation \emph{or} the MS formulation with $A=k^2$ are solved using the Galerkin method, then, given $k_0>0$, 
\begin{equation}\label{eq:3new2}
\N{u-u_N}_{\HokO}\leq \N{u-u_N}_{V_2}\lesssim h^2 k^3 (1+k^2h^2) \N{u}_{\HokO}\qquad \tfa k\ge k_0,
\end{equation}
i.e., the goal \eqref{eq:rel_error} is achieved if $h^2k^3$ is sufficiently small, i.e.~the method is $hk^{3/2}$-accurate (in the sense of Definition \ref{def:hka}).
\end{proposition}

\begin{proof}[Proof of Proposition \ref{lem:H1error}]
We describe the proof for $d=2$, the one-dimensional case follows along the same lines \cite[\S5]{BeBuRiSa:11}.
Let $\Pi:V\to V_h$ denote the projection operator defined as $\Pi_{\mathcal V}^K$ in \cite[page~301]{BeBuRiSa:11}.

Then, by \cite[Theorem~6]{BeBuRiSa:11} (where, in their notation, $k_1=k_2=k_*=k^*=2$ is given by the $C^1$ continuity of our spline basis, $p=3$ is the polynomial degree in each direction, and $\sigma=4$ is implied by $k_1+k_2\le \sigma\le p+1$), 
\begin{equation*}
\N{v-\Pi v}_{H^m(\Omega)}\lesssim h^{4-m}|v|_{H^4(\Omega)}, \quad m=0,1,2.
\end{equation*}
Combining with Lemma \ref{lem:VboundedBySobolev} and the oscillatory behaviour assumption \eqref{eq:1}, we find
\begin{align*}
\N{u-\Pi u}_{V_1} 
&\overset{\eqref{eq:VboundedBySobolev}}\lesssim \sum_{m=0}^2 k^{3/2-m}\N{u-\Pi u}_{H^m(\Omega)}\\
&\lesssim \sum_{m=0}^2 h^{4-m}k^{3/2-m}|u|_{H^4(\Omega)}
\nonumber\\
&\overset{\eqref{eq:1}}\lesssim \sum_{m=0}^2 h^{4-m}k^{9/2-m}\N{u}_\HokO
=h^2k^{5/2} (1+hk+h^2k^2) \N{u}_\HokO,
\nonumber
\end{align*}
and similarly
$$
\N{u-\Pi u}_{V_2} \lesssim h^2k^3 (1+hk+ h^2k^2) \N{u}_\HokO.
$$
Combining with the quasi-optimality results \eqref{eq:qo1} and \eqref{eq:qo2}, we obtain the assertion.
\end{proof}

\begin{figure}[htb]
\includegraphics[width=74mm,clip,trim = 20 0 40 0]{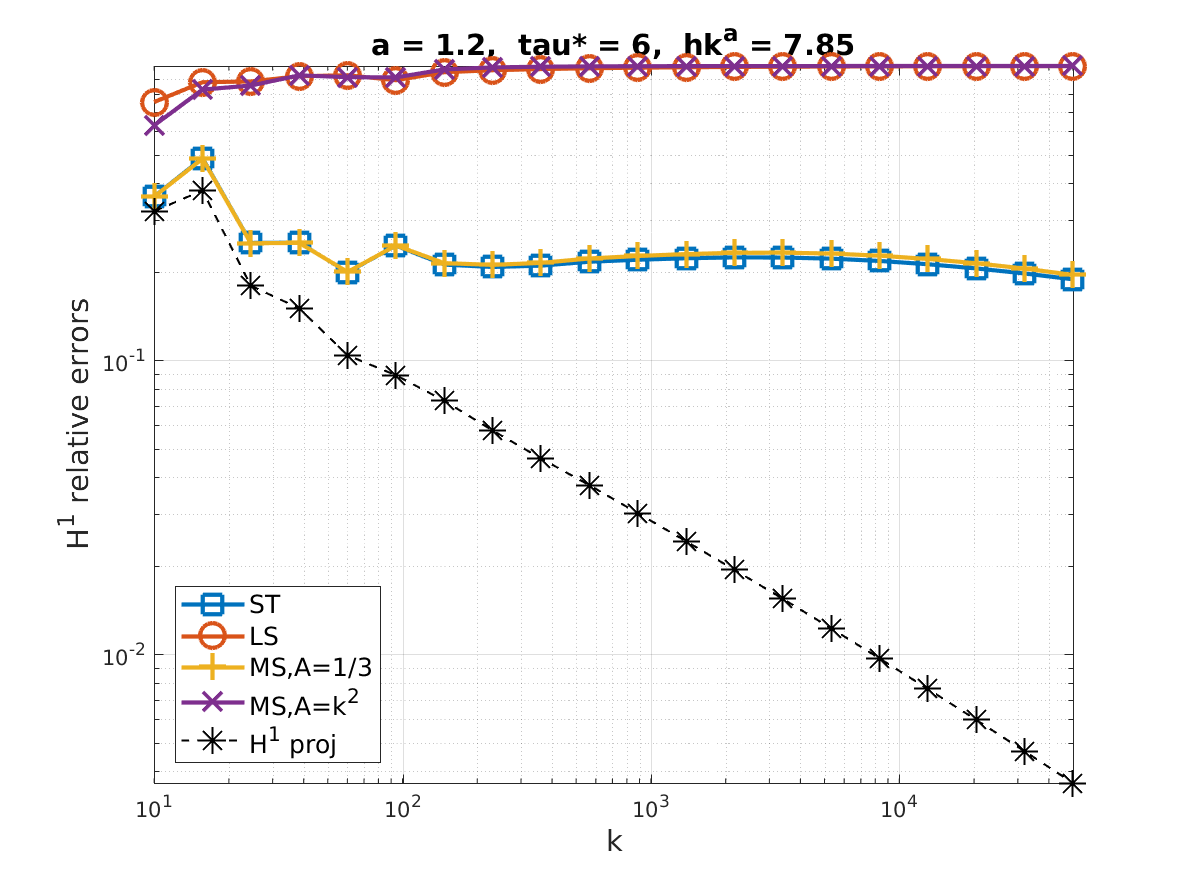}
\includegraphics[width=74mm,clip,trim = 20 0 40 0]{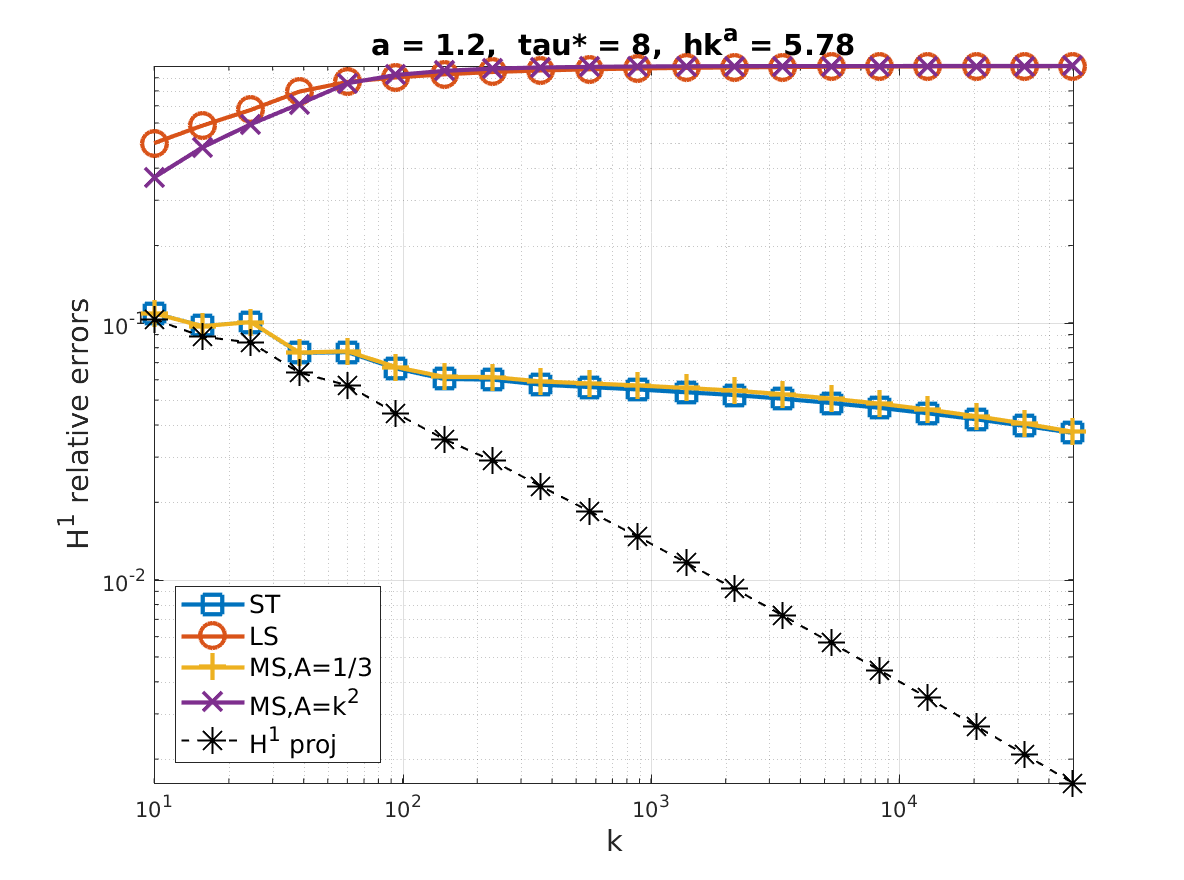}\\
\includegraphics[width=74mm,clip,trim = 20 0 40 0]{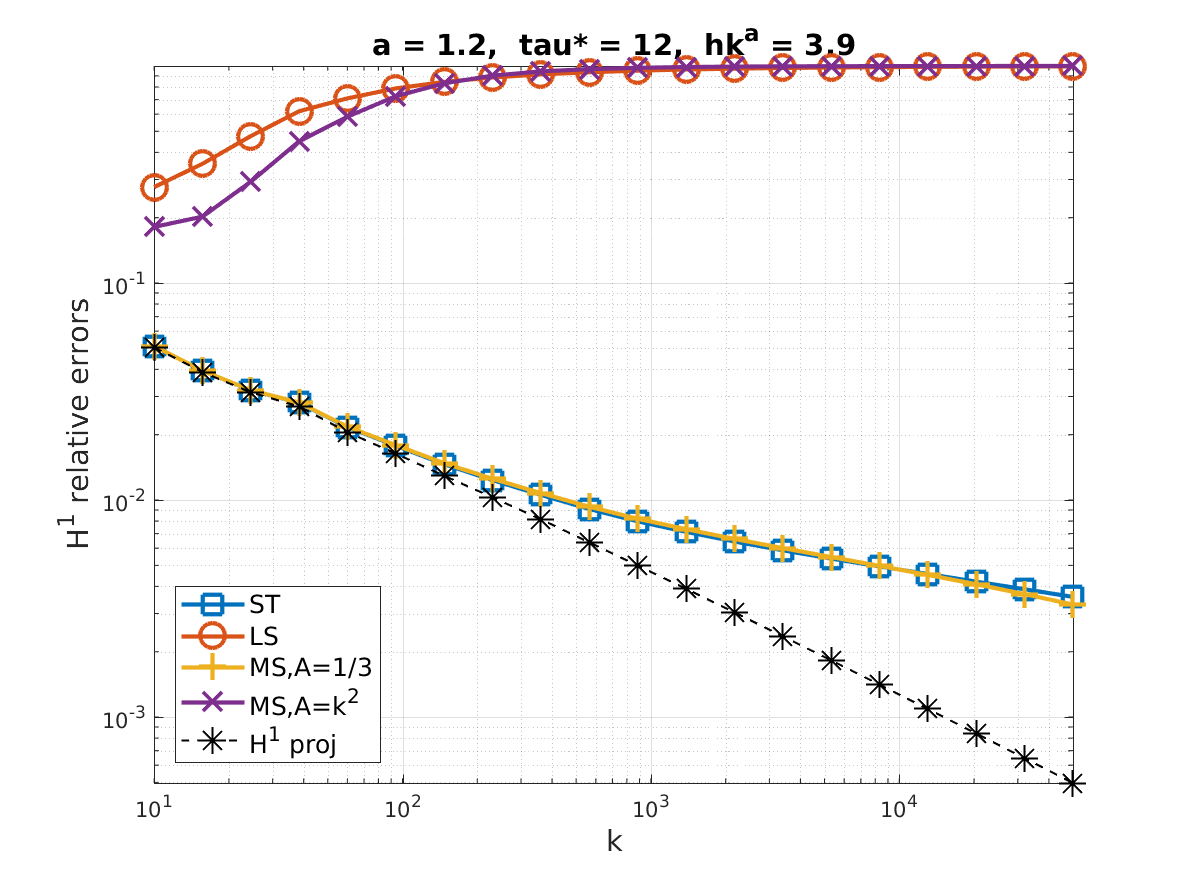}   
\includegraphics[width=74mm,clip,trim = 20 0 40 0]{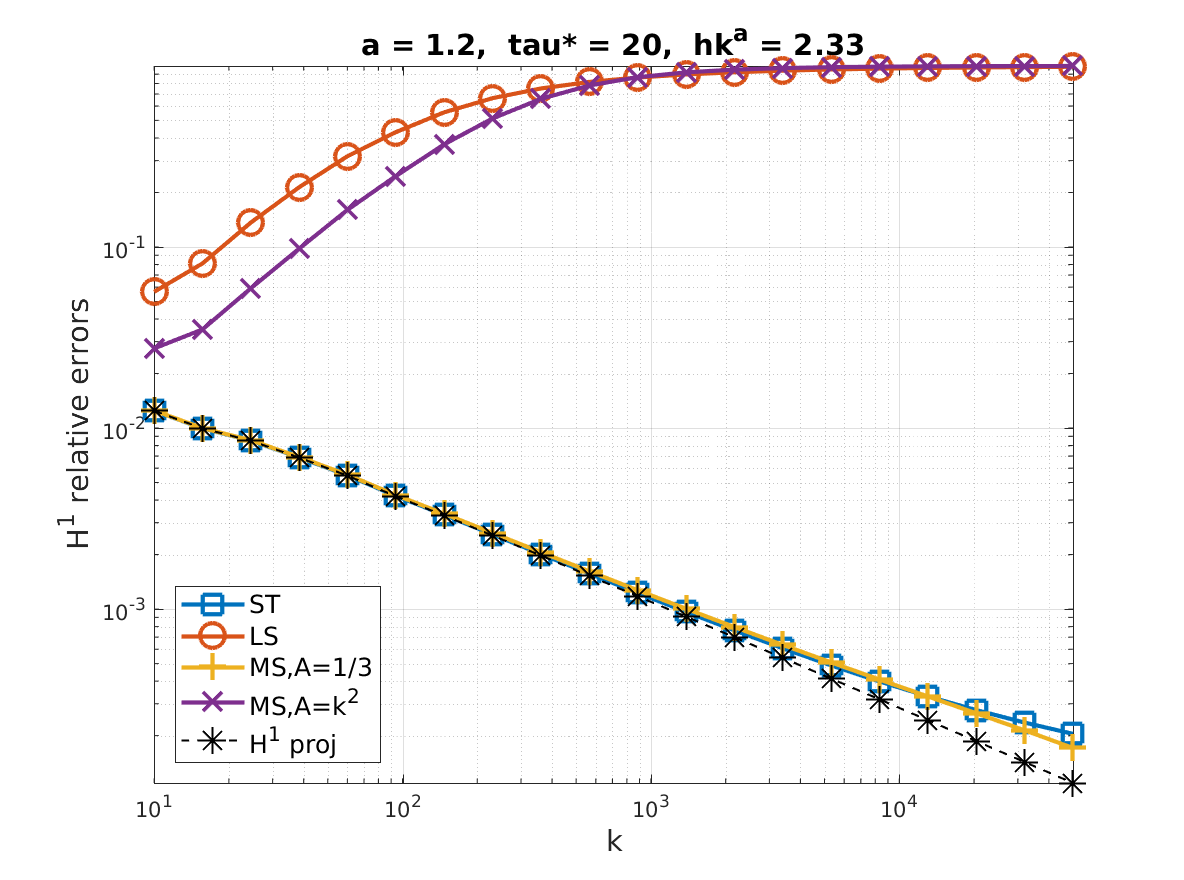} 
\caption{The relative $H^1_k(\Omega)$ errors of the standard formulation, the least squares formulation, the MS formulation with $A=1/3$ and with $A=k^2$ for problem \eqref{eq:bvp_main} in $\Omega=(0,1)$ with $u(x)=\re^{\ri k x}$.
The $H^1_k(\Omega)$ best-approximation relative error is represented by the dashed black lines with stars.
Here $k$ runs from 10 to 50\,000 and  is chosen such that $hk^{6/5}=C$ for four different values of $C$, ensuring that $\tau^*=6,8,12,20$.
\newline
The error of the MS formulation with $A=1/3$ (yellow ``+'' signs) and that of the standard formulation (blue squares) is uniformly bounded, so these appear to be $hk^{6/5}$-accurate.
The relative error of the least squares formulation (red circles) and the MS formulation with $A=k^2$ (purple crosses) quickly reach 100\% for all choices of $\tau^*$, so these appear not to be $hk^{6/5}$-accurate.}
\label{fig:Alpha1.2}
\end{figure}

\begin{figure}[htbp]
\includegraphics[width=74mm,clip,trim = 20 0 40 0]{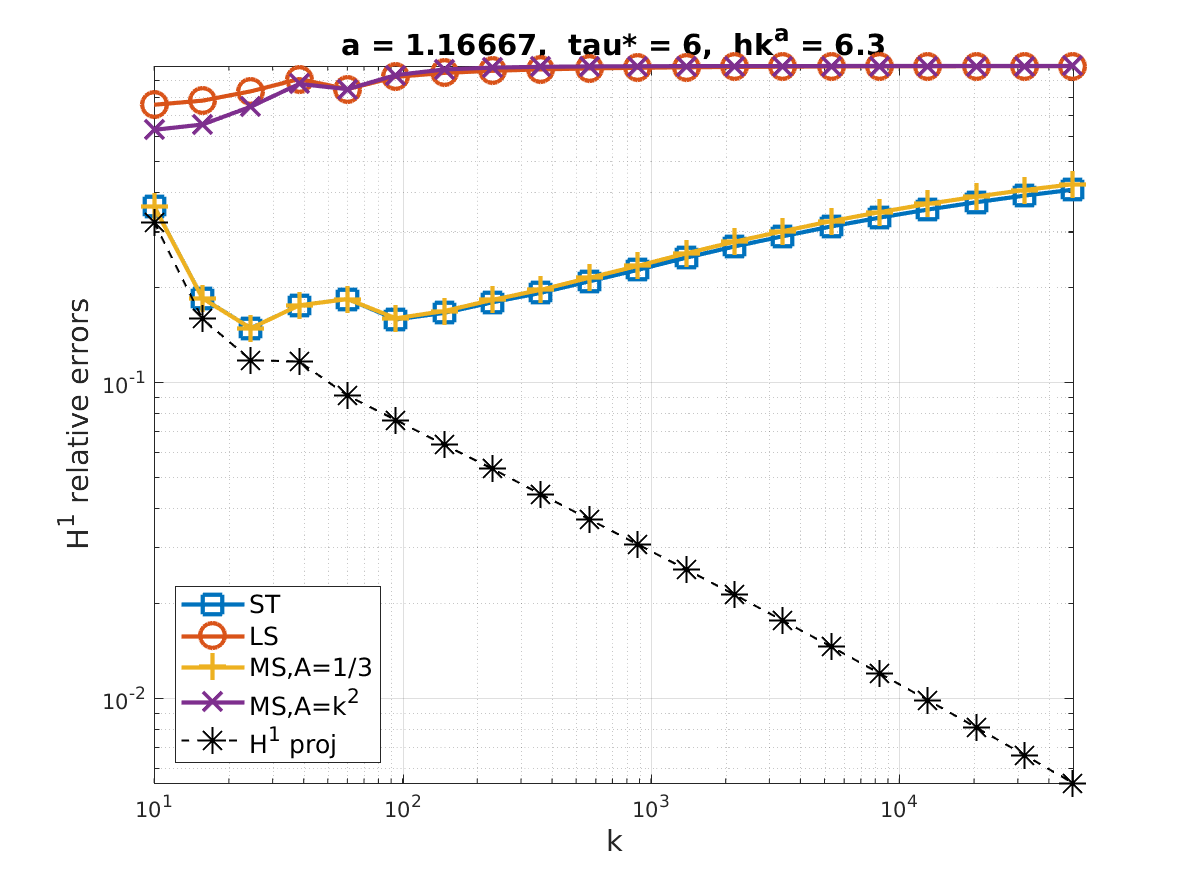}
\includegraphics[width=74mm,clip,trim = 20 0 40 0]{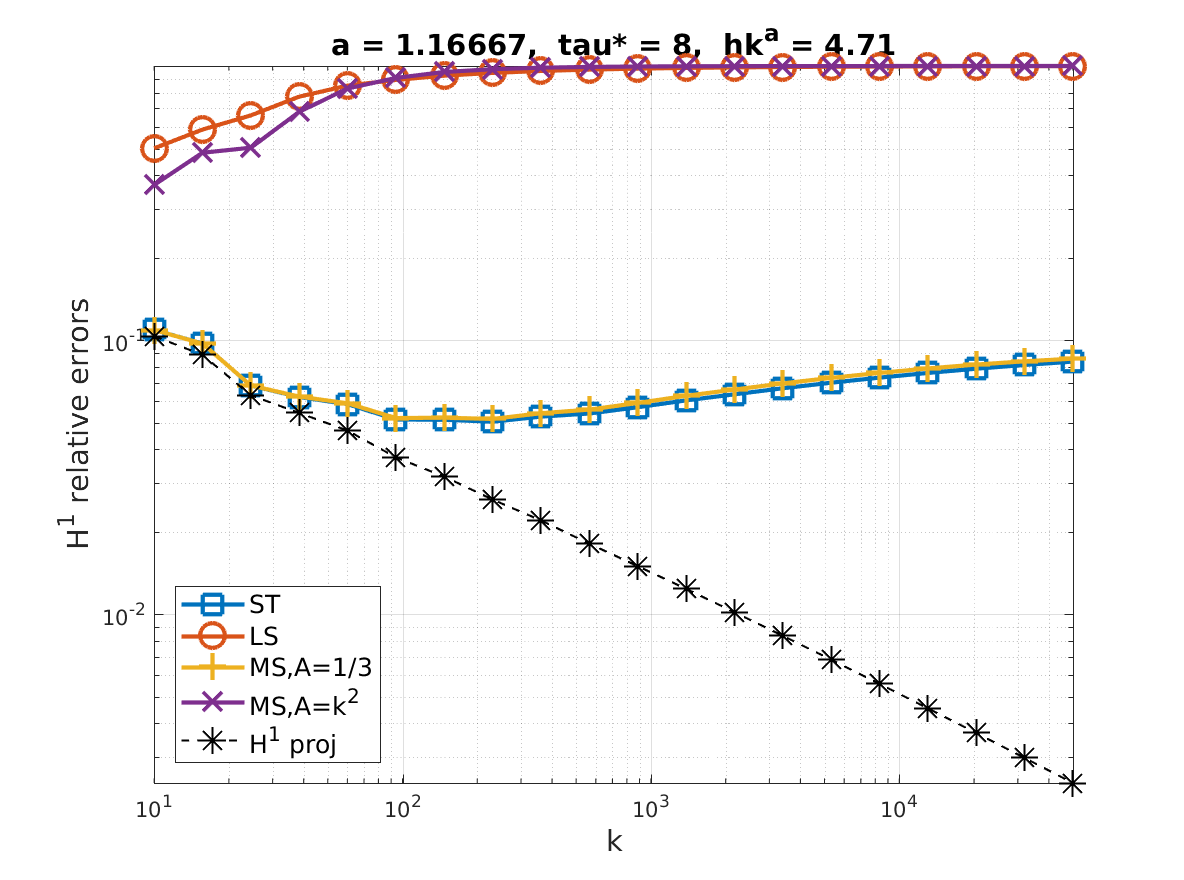}\\
\includegraphics[width=74mm,clip,trim = 20 0 40 0]{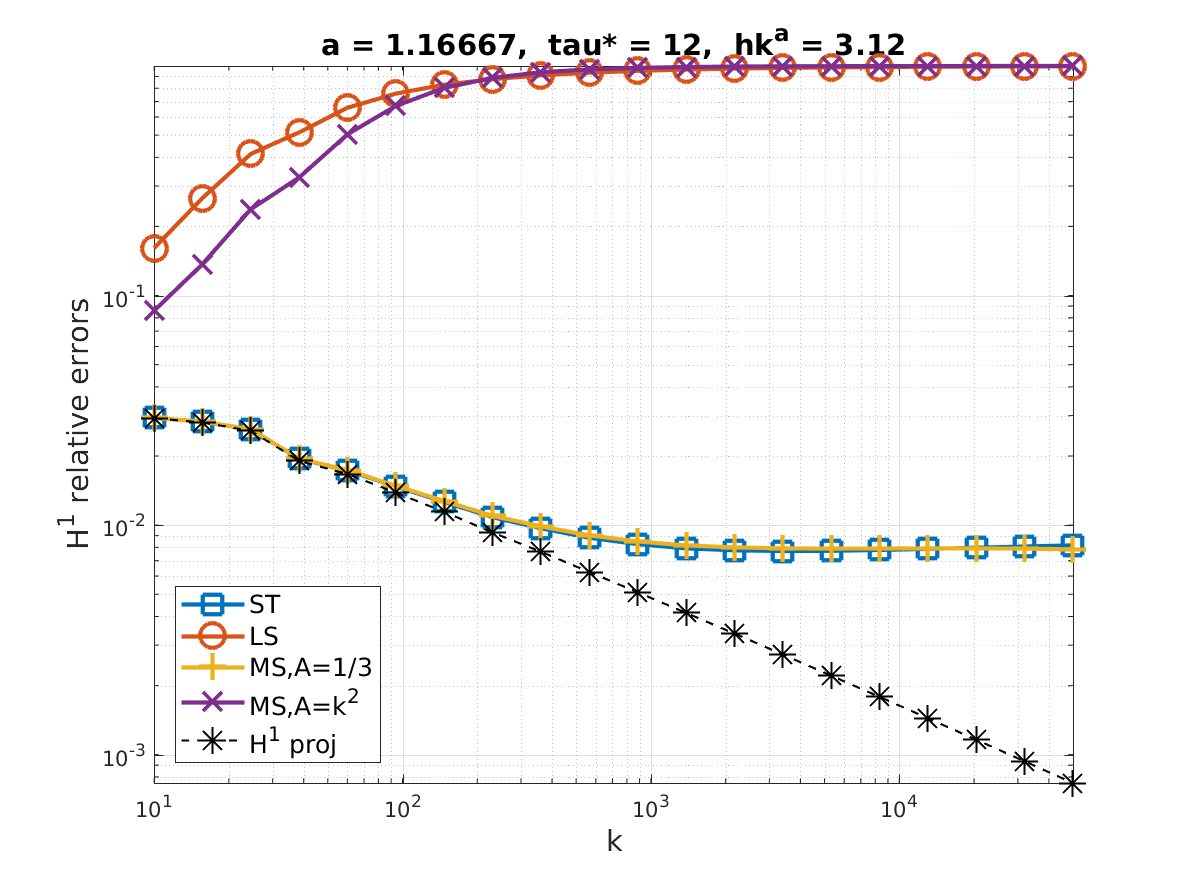}
\includegraphics[width=74mm,clip,trim = 20 0 40 0]{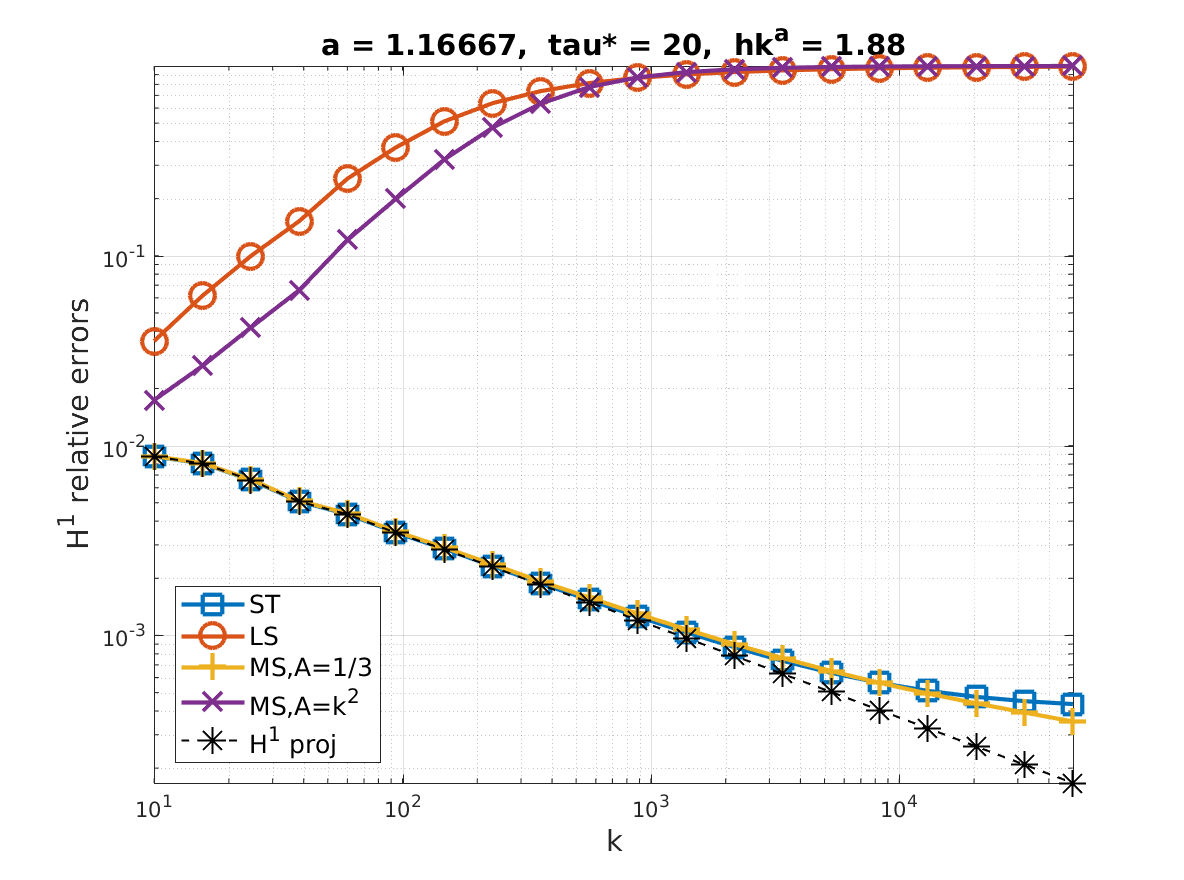}
\caption{Same as in Figure \ref{fig:Alpha1.2} with $hk^{7/6}=C$.
In this case the MS formulation with $A=1/3$ and the standard formulation appear not to be uniformly bounded in $k$.
For large values of $\tau^*$ it appears that a long preasymptotic regime is present and the relative errors grow only for large values of $k$.
}
\label{fig:Alpha7/6}
\end{figure}

\begin{figure}[htbp]
\includegraphics[width=74mm,clip,trim = 20 0 40 0]{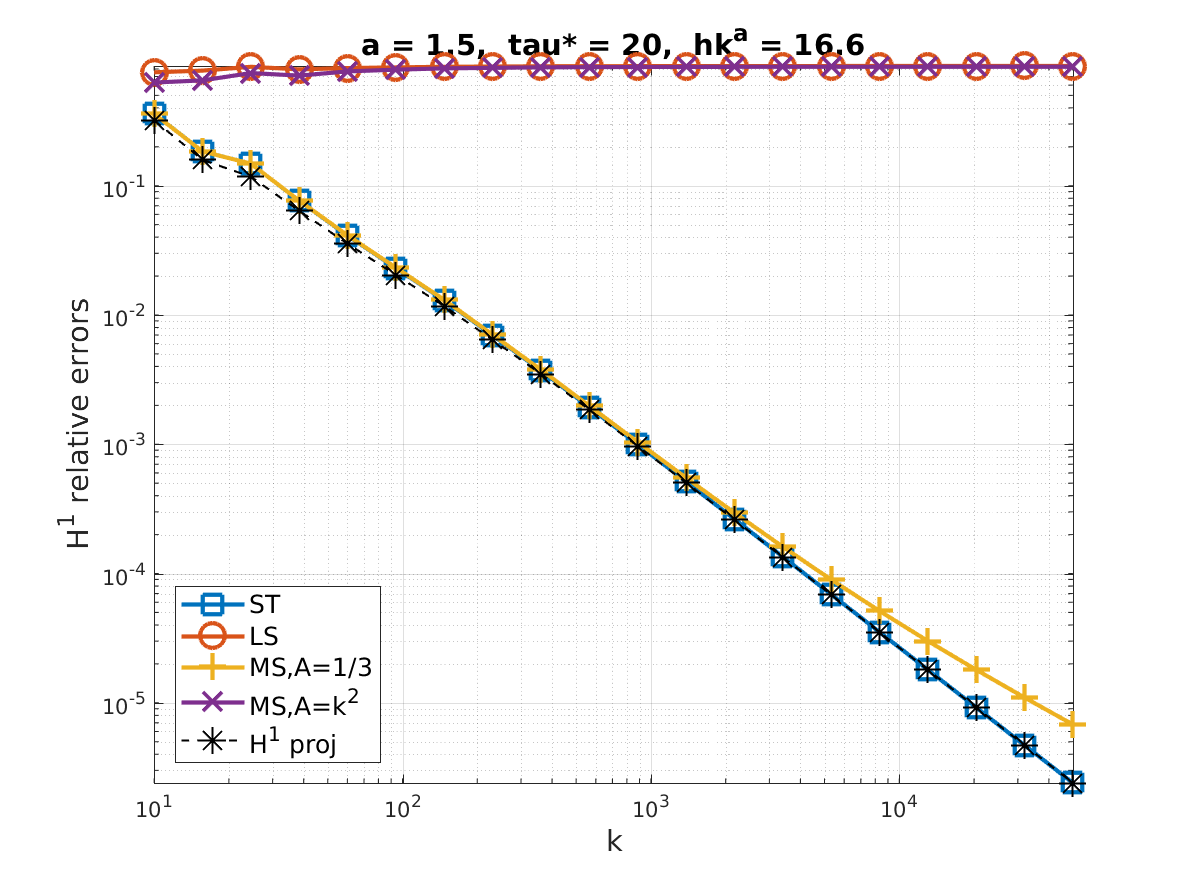} 
\includegraphics[width=74mm,clip,trim = 20 0 40 0]{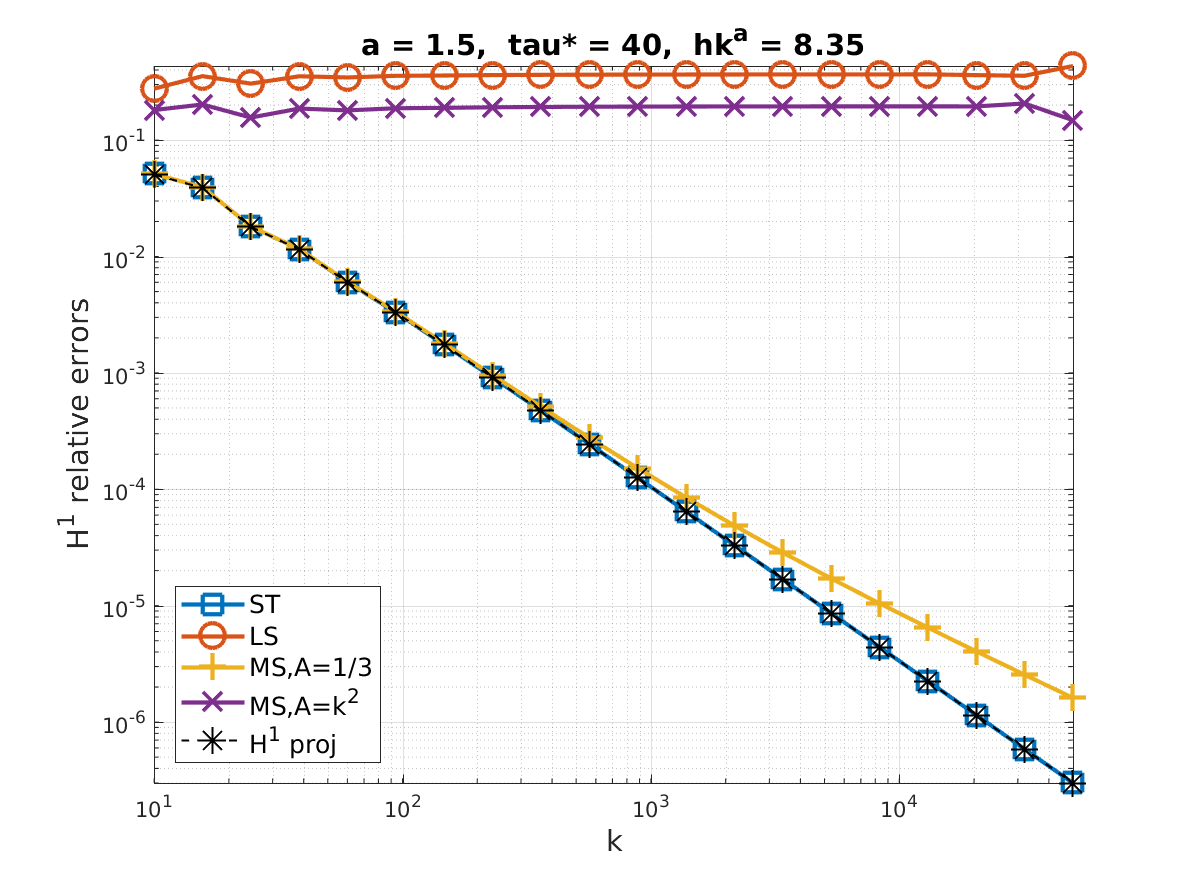}
\caption{Same as in Figure \ref{fig:Alpha1.2} with $hk^{3/2}=C$ and $\tau^*=20$ and $40$.
For $\tau^*=40$ the relative errors of least squares formulation and the MS formulation with $A=k^2$ are bounded by $0.44$ and $0.2$, respectively.
These two formulations appears to be $hk^{3/2}$-accurate.
}
\label{fig:Alpha1.5}
\end{figure}

\begin{figure}[htbp]
\includegraphics[width=74mm,clip,trim = 40 0 40 0]{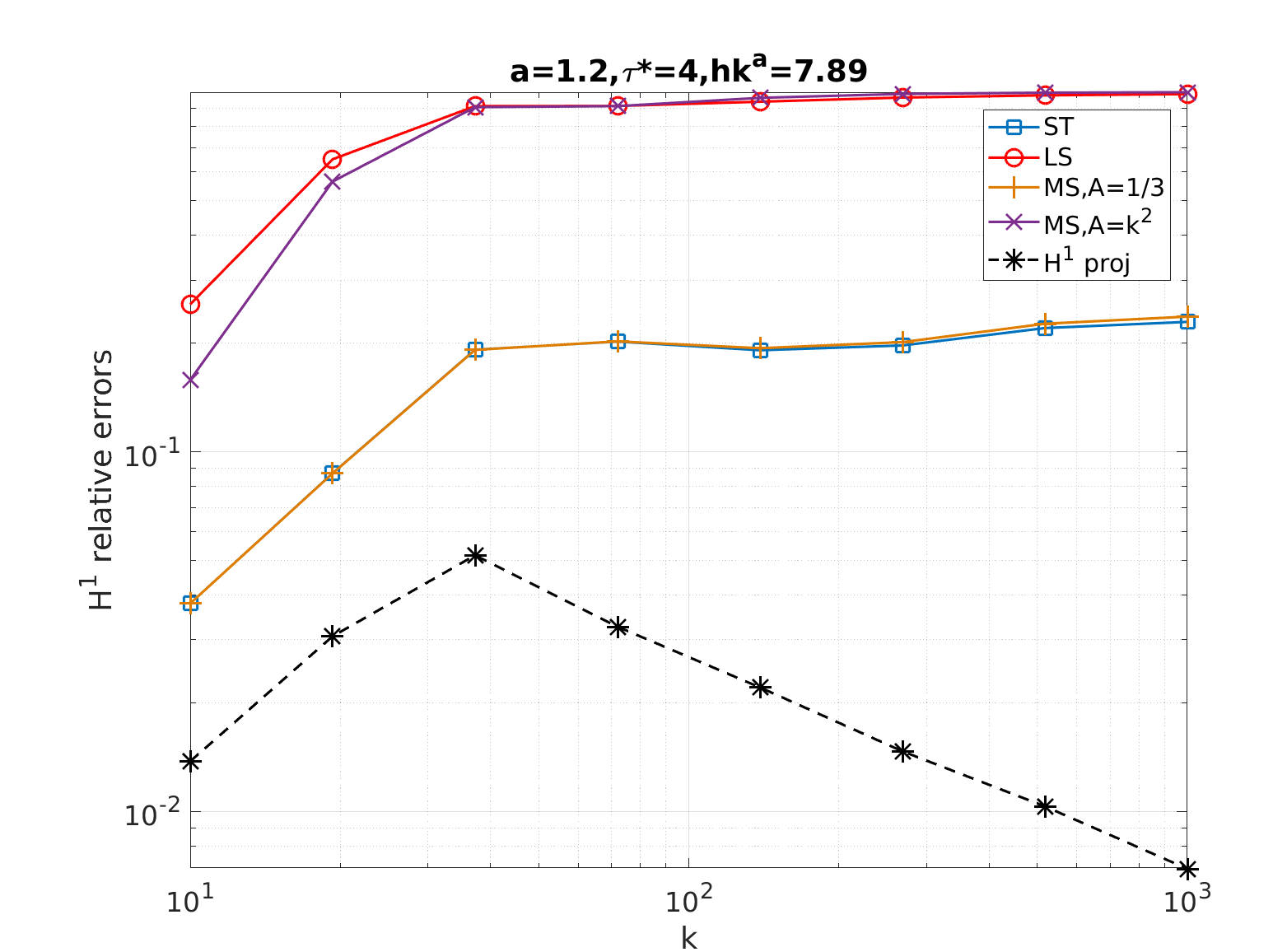}
\includegraphics[width=74mm,clip,trim = 40 0 40 0]{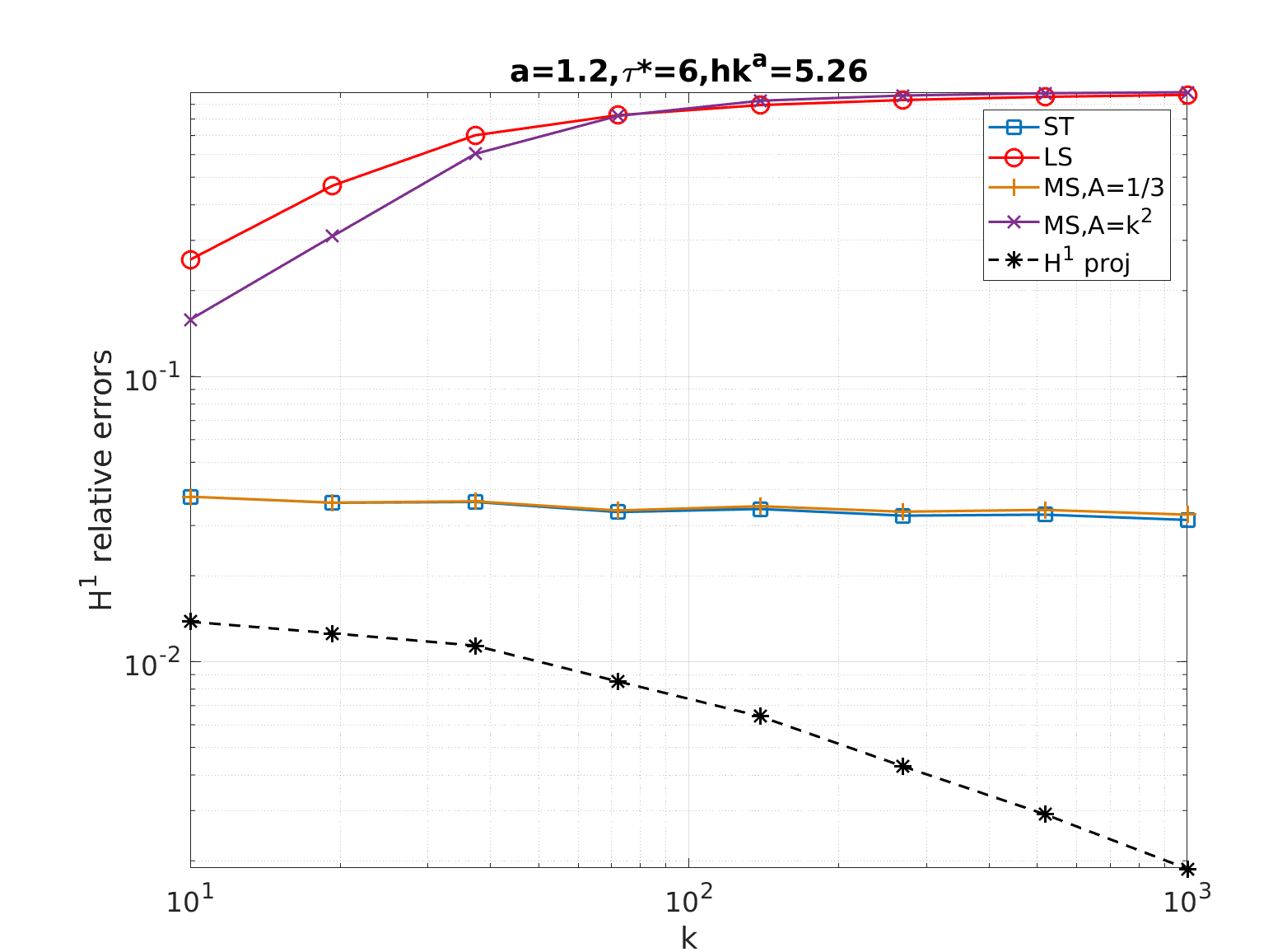}
\caption{The 2-d analogue of Figure \ref{fig:Alpha1.2} with $hk^{6/5}=C$ and $\tau^*=4$ (left) and $6$ (right). 
As in the 1-d case, the standard formulation and the MS formulation with $A=1/3$ appear to be $hk^{6/5}$-accurate.}
\label{fig:direct2-d}
\end{figure}

\subsection{Numerical experiments}\label{sec:acc_num}

We first describe experiments for the IIP \eqref{eq:bvp_main} in 1-d, with $\Omega=(0,1)$, $f=0$, and $g$ chosen so that the exact solution is $\re^{\ri k x}$.
As stated in \S\ref{sec:Vn} we discretise all the formulations with Hermite elements ($C^1$ cubics) and uniform meshes.
Figures \ref{fig:Alpha1.2}, \ref{fig:Alpha7/6} and \ref{fig:Alpha1.5} display the relative $H^1_k(\Omega)$ errors for $k$ from $k_{\min}=10$ to $k_{\max}=50\,000$ (note that the numerical experiments in \cite{IhBa:95a} go up to $k=1\,000$), where $h$ is tied to $k$ via $hk^a=C$.
The three figures correspond to $a=6/5$, $a=7/6$ and $a=3/2$ respectively, while the  subpanels correspond to different values of $C=4\pi(k_{\min}k_{\max})^{\frac{a-1}2}/\tau^*$, 
where $\tau^*$  represents approximately the number of degrees of freedom per wavelength in the central experiment of the (loglog) plot.

Our conclusions from these 1-d experiments are the following:
\begin{enumerate}
\item The standard formulation and the MS formulation with $A=1/3$ appear to be $hk^{6/5}$-accurate.
\item The least-squares formulation  and the MS formulation with $A=k^2$ appear to be $hk^{3/2}$-accurate.
\end{enumerate}

Regarding 1: 
The 1-d results of Ihlenburg and Babu\v{s}ka recapped in \S\ref{sec:accurate} imply that the $h$-FEM is $hk^a$accurate when
$a=(2p+1)/(2p)$, i.e.~when $p=3$, the $h$-FEM is $hk^{7/6}$-accurate. 
However, these results are only for $C^0$ elements and not for $C^1$ elements.
Figure \ref{fig:Alpha7/6} shows the relative errors of the standard and MS formulations growing when $hk^{7/6}=C$, for a sufficiently large value of $C$.
In Proposition \ref{lem:H1error} we could only prove that the MS formulation with $A=1/3$ is $hk^{7/4}$ accurate (although, in contrast to the results of \cite{IhBa:95, IhBa:97}, this proof holds for $d\geq 1$).  
If the linear dependence on $k$ of the quasi-optimality constant were replaced by $k^{0.4}$ (the $k$-dependence that Figure \ref{fig:QuasiOptSurface} indicates is sharp), then the MS formulation with $A=1/3$ would be provably $hk^{1.45}$-accurate.

Regarding 2:
this is in agreement with the result of Proposition \ref{lem:H1error}.

\

The 2-d analogues of Figures \ref{fig:Alpha1.2}, \ref{fig:Alpha7/6}, and \ref{fig:Alpha1.5} all exhibit the same behaviour as in 1-d (at least up to $k=1\,000$). 
We only display the 2-d analogue of Figure \ref{fig:Alpha1.2}, Figure \ref{fig:direct2-d}, where we see that, just as in 1-d, the standard formulation and the MS formulation with $A=1/3$ appear to be $hk^{6/5}$-accurate. For all these 2-d experiments, we choose $\Omega=(0,1)^2$, $f=0$, and $g$ so that the exact solution is $\re^{\ri k x}$.

\section{Iterative solution}\label{sec:test_fast}

We first recap the GMRES theory based on the field of values (\S\ref{sec:GMRES}), apply it to the MS and LS formulations (\S\ref{sec:GMRES_theorems}), and give some numerical experiments (\S\ref{sec:fast_num}).

\subsection{Recap of GMRES convergence theory}\label{sec:GMRES}

We now recap the GMRES convergence theory based on the field of values/numerical range, originally due to Elman \cite{El:82} and improved by Beckermann, Goreinov, and Tyrtyshnikov \cite{BeGoTy:06}. We give this theory for weighted GMRES; the theory for the standard, unweighted GMRES follows by setting the weight matrix ($\matrixD$ below) equal to the identity.

We consider the abstract  linear system 
 \begin{equation*}
\matrixC \bfx = \bfd 
\end{equation*}
in $\mathbb{C}^n$, where $\matrixC$ is an  $n\times n$ nonsingular complex matrix.   
Given an initial guess $\bfx^0$, we introduce the residual $\bfr^0 = \bfd- \matrixC \bfx^0$ and 
the usual Krylov spaces:  
$$  \cK^m(\matrixC, \bfr^0) := \mathrm{span}\{\matrixC^j \bfr^0 : j = 0, \ldots, m-1\}  .$$
Let $\langle \cdot , \cdot \rangle_{\matrixD}$ denote the inner product on $\C^n$ 
induced by some Hermitian positive-definite  matrix $D$, i.e.  
\begin{equation*}
\langle \bV, \bW\rangle_{\matrixD} := \bW^* \matrixD \bV
\end{equation*} 
with induced norm $\Vert \cdot \Vert_{\matrixD}$, where $^*$ denotes Hermitian transpose. For $m \geq 1$, define   $\bfx^m$  to be  the unique element of $\cK^m$ satisfying  the   
 minimal residual  property: 
$$  \Vert \bfr^m \Vert_{\matrixD} := \Vert \bfd - \matrixC \bfx^m \Vert_{\matrixD} =  \min_{\bfx \in \cK^m(\matrixC, \br^0)} \Vert {\bfd} - {\matrixC} {\bfx} \Vert_{\matrixD}. $$
When $\matrixD = \matrixI$ this is just the usual GMRES algorithm, and we use $\Vert \cdot \Vert$ to denote  $\Vert \cdot \Vert_I$, but for  more general  $\matrixD$ it 
is the weighted GMRES method \cite{Es:98} in which case  
its implementation requires the application of the weighted Arnoldi process \cite{GuPe:14}.

The following theorem is a simple generalisation to the weighted setting of the GMRES convergence result of Beckermann, 
Goreinov, and Tyrtyshnikov \cite{BeGoTy:06}. This result is an improvement of the so-called ``Elman estimate'', originally due to Elman \cite{El:82}; see also \cite{EiElSc:83}, \cite[Theorem 3.2]{St:97}, \cite[Corollary 6.2]{EiEr:01}, \cite{LiTi:12}, and the review \cite[\S6]{SiSz:07}.

\begin{theorem}[Elman-type estimate for weighted GMRES] \label{thm:Elman}  
Let $\matrixC$ be a matrix with $0\notin W_{\matrixD}(\matrixC)$, where 
$$
W_{\matrixD}(\matrixC):= \big\{ \langle \matrixC \bv, \bv\rangle_{\matrixD} : \bv \in \Com^N, \|\bv\|_{\matrixD}=1\big\}
$$
is the \emph{field of values}, also called the \emph{numerical range} of $\matrixC$ with respect to the inner product $\langle\cdot,\cdot\rangle_{\matrixD}$. 
Let $\sigma\in [0,\pi/2)$ be defined such that
\begin{equation}\label{eq:cosbeta}
\cos \sigma = \frac{\mathrm{dist}\big(0, W_{\matrixD}(\matrixC)\big)}{\| \matrixC\|_{\matrixD}},
\end{equation}
let $\gamma_\sigma$ be defined by 
$$
\gamma_\sigma:= 2 \sin \left( \frac{\sigma}{4-2\sigma/\pi}\right),
$$
and let $\br_m$ be defined as above.
Then
\begin{equation}\label{eq:Elman2}
\frac{\|\br_m\|_{\matrixD}}{\|\br_0\|_{\matrixD}} \leq \left(2 + \frac{2}{\sqrt{3}}\right)\big(2+ \gamma_\sigma\big) \,\gamma_\sigma^m.
\end{equation}
\end{theorem}

\begin{proof}[References for the proof of Theorem \ref{thm:Elman}]
Theorem \ref{thm:Elman} is proved in \cite[Theorem 5.3]{BoDoGrSpTo:17} using 
\cite[Theorem 2.1]{BeGoTy:06} and \cite[Theorem 5.1]{GrSpVa:17}.
\end{proof}

We apply Theorem \ref{thm:Elman} below to a situation where $\cos\sigma \tendo$ as $k\tendi$. Since $\sigma\in [0,\pi/2)$, 
the limit $\cos\sigma\tendo$ corresponds to the limit $\sigma \rightarrow \pi/2$, and it is therefore convenient to summarise Theorem \ref{thm:Elman} applied to this setting as the following corollary.

\begin{corollary}\label{cor:Elman}
With $\matrixC$ a matrix such that $0\notin W_{\matrixD}(\matrixC)$, let $\epsilon\in (0,\pi/2]$ be defined such that
$$
\sin \epsilon = \frac{\mathrm{dist}\big(0, W_{\matrixD}(\matrixC)\big)}{\| \matrixC\|_{\matrixD}}
$$
i.e.~$\epsilon=\pi/2-\sigma$ where $\sigma$ is defined by \eqref{eq:cosbeta}.
There exists $C>0$ (independent of $\epsilon$) such that, given $0<\delta<1$, 
$$
\text{if} \quad m\geq \frac{C}{\epsilon}\log\left(\frac{12}{\delta}\right)\quad\text{ then }  \quad\frac{\|\br_m\|_{\matrixD}}{\|\br_0\|_{\matrixD}}\leq \delta.
$$
\end{corollary}

\noi Corollary \ref{cor:Elman} is proved in \cite[Corollary 5.4]{BoDoGrSpTo:17}, and implies that choosing $m\gtrsim \epsilon^{-1}$ is sufficient for 
the decrease of the residual to be independent of $\epsilon$ as $\epsilon\tendo$.

\begin{remark}[Comparison with the original Elman estimate]
The original bound proved by Elman (in the unweighted setting) is
\begin{equation}\label{eq:Elman}
\frac{\|\br_m\|_{\matrixD}}{\|\br_0\|_{\matrixD}} \leq \sin^m \sigma.
\end{equation}
To see that \eqref{eq:Elman2} is a stronger result, observe that, when $\sigma= \pi/2-\epsilon$, the convergence factor in  \eqref{eq:Elman} is
$$
\sin\sigma = \cos\epsilon 
= 1- \frac{\epsilon^2}{2}+\cO(\epsilon^4),
$$
which leads to requiring $m\gtrsim \epsilon^{-2}$ for GMRES to converge in an $\epsilon$-independent way as $\epsilon\tendo$. In contrast, the convergence factor in \eqref{eq:Elman2} is 
$$
\gamma_\sigma:=2 \sin \left( \frac{\sigma}{4-2\sigma/\pi}\right)= 2 \sin \left( \frac{\pi}{6} -\frac{4\epsilon}{9}+\cO(\epsilon^2)\right)= 1 -\frac{4\epsilon}{3\sqrt{3}} + \cO(\epsilon^2)\quad\tas\,\, \epsilon\tendo,
$$
leading to $m\gtrsim \epsilon^{-1}$ as stated in Corollary \ref{cor:Elman}.
\end{remark}

\subsection{The theory applied to the MS and LS formulation}\label{sec:GMRES_theorems}

We now apply the theory in \S\ref{sec:acc_num} to the MS formulation; all the results for MS with $A=k^2$ also hold for LS formulation (similar to in Lemma \ref{lem:QuasiOpt} and Proposition \ref{lem:H1error})
because the $k$-dependence of $\Ccont/\Ccoer$ is the same for the two formulations (see Lemma \ref{lem:LSCC} and Corollary \ref{cor:MS}).

In \cite[\S5.2]{MoSp:14} the implications of the original Elman estimate \eqref{eq:Elman} were explored theoretically for the MS formulation using GMRES in the standard $l^2$ inner product, and computations were done for standard GMRES in \cite{GaMo:17a} (see the discussion in \S\ref{rem:GM1}).
We highlight that conforming discretisations of the standard formulation \eqref{eq:vfH1} are not amenable to the analysis in \S\ref{sec:GMRES}, since, by Part {\em(iii)} of Lemma~\ref{lem:contcoer_ST}, 
the distance between the numerical range and the origin converges to zero when the discretisation is refined.

In this section, we explore the implications of the refined version of the Elman estimate \eqref{eq:Elman2} for weighted GMRES, where the weight matrix corresponds to the mass matrix of one of the norms $\N{\cdot}_{V_1}$ and $\N{\cdot}_{V_2}$. 
This weighted setting was inspired by the recent work on 
domain-decomposition preconditioners for the Helmholtz and Maxwell equations 
in \cite{GrSpVa:17} and \cite{BoDoGrSpTo:17}.

We first need to set up some notation for the Galerkin method applied to the MS formulation \eqref{eq:coercive}.
From now on we assume that $L=1$ in the definition of the norms \eqref{eq:normV}--\eqref{eq:Valt}.

\paragraph{Notation for the matrices involved in the Galerkin method.}

Let the real ($C^1$) basis functions be denoted by $\phi_j$.
Recall from \eqref{eq:discrete} that 
\begin{equation*}
\matrixS_{\ell,m} := \int_{\Omega} \nabla \phi_\ell \cdot \nabla\phi_m \, \rd \bx, \qquad   
\matrixM_{\ell,m} := \int_{\Omega}\phi_\ell  \,\phi_m \, \rd \bx, \qquad     
\matrixNo_{\ell,m} := \int_{\Gamma} \phi_\ell \, \phi_m  \rd s.
\end{equation*}
Define
\begin{align}
\matrixL^{(1)}_{\ell,m} :=&\; \int_{\Omega} \Delta \phi_\ell \,\Delta \phi_m, \qquad 
\matrixL^{(2)}_{\ell,m} = \int_{\Omega} (\Delta \phi_\ell+ k^2\phi_\ell) (\Delta \phi_m+ k^2\phi_m),\nonumber\\
\matrixNone_{\ell,m} :=& \;\int_{\Gamma} \nT\phi_\ell\cdot\nT  \phi_m, \qquad
\matrixNtwo_{\ell,m}= \int_{\Gamma} \partial_n\phi_\ell\,\partial_n \phi_m\nonumber\\
\label{eq:Dk1}
\matrixD_k^{(1)}:=& \;\frac{1}{k^2}\matrixL^{(1)} + \matrixS + k^2\matrixM + k^2\matrixNo+ \matrixNone+\matrixNtwo \quad \tand\\
\matrixD_k^{(2)}:=& \;\matrixL^{(2)} + \matrixS + k^2\matrixM + k^2\matrixNo+ \matrixNone+\matrixNtwo.\label{eq:Dk2}
\end{align}
With $\matrixP$ a Hermitian positive definite matrix in $\C^{N\times N}$, we denote the corresponding scalar product and norm
\begin{equation}\label{eq:Dkip}
\langle \bv,\bw\rangle_\matrixP :=\bw^* \matrixP \bv = (\matrixP \bv, \bw)_2,
\qquad 
\N{\bv}_\matrixP^2:= \langle \bv,\bv\rangle_\matrixP\qquad \bv,\bw\in\C^N,
\end{equation}
where $(\cdot,\cdot)_2$ denotes the Euclidean $l^2$ inner product and $\bw^*$ the conjugate transpose of $\bw$.
These definitions and the definitions of the norms $\N{\cdot}_{V_1}$ \eqref{eq:normV} and $\N{\cdot}_{V_2}$ \eqref{eq:Valt} imply that
\begin{equation}\label{eq:normequivDk}
\N{v_h}_{V_1} = \N{\bv}_{\matrixD_k^{(1)}}\quad\tand\quad \N{v_h}_{V_2} = \N{\bv}_{\matrixD_k^{(2)}},
\end{equation}
where $\bv$ is the coefficient vectors of $v_h\in V_N$.
Observe that $\matrixD_k^{(j)}$, $j=1,2,$ are real, symmetric, and positive definite (we have $\bv^* \matrixD_k^{(j)}\bv>0$ for all $\bv \in \Com^N\setminus\{\bze\}$ because $\N{\cdot}_{V_1}$ and $\N{\cdot}_{V_2}$ are norms).

Define
\begin{equation}\label{eq:Bg}
\matrixB^{(1)}_{i j} := b(\phi_j, \phi_i) \quad\tand\quad g^{(1)}_i = G(\phi_i),
\end{equation}
where $b(\cdot,\cdot)$ and $G(\cdot)$ are defined by \eqref{eq:B} and \eqref{eq:F} respectively and $A=1/3$.
Similarly, 
define
\begin{equation}\label{eq:Bg2}
\matrixB^{(2)}_{i j} := b(\phi_j, \phi_i) \quad\tand\quad g^{(2)}_i = G(\phi_i),
\end{equation}
where $b(\cdot,\cdot)$ and $G(\cdot)$ are defined by \eqref{eq:B} and \eqref{eq:F} respectively and $A=k^2$.
The definition of $\matrixB^{(j)}$ implies that
$$
b(u_h, v_h) = \left(\matrixB^{(j)} \bu,\bv\right)_2
$$
with $A=1/3$ when $j=1$ and $A=k^2$ when $j=2$.
The linear system arising from the Galerkin method applied to the new variational formulation is then $\matrixB^{(j)} \bu= \bg^{(j)}$.

\begin{lemma}[Continuity and coercivity in $(\cdot,\cdot)_{\matrixD_k}$ for left-preconditioned system]\label{lem:contcoer_left}\ 

\noindent
Assume that $\Omega$ is star-shaped with respect to a ball and $\beta$ satisfies both \eqref{eq:beta2} and \eqref{eq:beta1}.
Let $\Ccontj$ and $\Ccoerj$ be the continuity and coercivity constants of $b(\cdot,\cdot)$ in the norm $\N{\cdot}_{V_j}$ 
for $A=1/3$ when $j=1$ and for $A=k^2$ when $j=2$.
Then 
\begin{equation}\label{eq:cont_coer_Dk}
\!\left|\left\langle\! \left(\matrixD_k^{(j)}\right)^{-1}\!\matrixB^{(j)} \bv,\bw\right\rangle_{\matrixD_k^{(j)}}\right| \leq \Ccontj\N{\bv}_{\matrixD_k^{(j)}}\N{\bw}_{\matrixD_k^{(j)}}
\;\tand \;
\left|\left\langle\! \left(\matrixD_k^{(j)}\right)^{-1}\!\matrixB^{(j)} \bv,\bv\right\rangle_{\matrixD_k^{(j)}}\right| \geq \Ccoerj \N{\bv}^2_{\matrixD_k^{(j)}}
\end{equation}
for all $\bv,\bw\in \Com^N$.
\end{lemma}

\begin{proof}
The definitions of continuity and coercivity in the $\N{\cdot}_{V_1}$ and $\N{\cdot}_{V_2}$ norms, the definition of $\matrixB^{(j)}$ \eqref{eq:Bg}/\eqref{eq:Bg2}, and the norm-equivalence \eqref{eq:normequivDk} imply that 
\begin{equation}\label{eq:10}
\left|\left(\matrixB^{(j)} \bv,\bw\right)_2\right| \leq \Ccontj \N{\bv}_{\matrixD_k^{(j)}}\N{\bw}_{\matrixD_k^{(j)}} \quad \left|\left(\matrixB^{(j)} \bv,\bv\right)_2\right| \geq \Ccoerj \N{\bv}^2_{\matrixD_k^{(j)}}.
\end{equation}
The results then follow from the fact that
$$
\left(\matrixB^{(j)} \bv,\bw\right)_2= \left( \matrixD_k^{(j)} \left(\matrixD_k^{(j)}\right)^{-1} \matrixB^{(j)}\bv, \bw\right)_2= \left\langle \left(\matrixD_k^{(j)}\right)^{-1} \matrixB^{(j)}\bv, \bw\right\rangle_{\matrixD_k^{(j)}}.
$$
\end{proof}

\begin{lemma}[Continuity and coercivity in $(\cdot,\cdot)_{\matrixD_k^{-1}}$ for right-preconditioned system]\label{lem:contcoer_right}\ 

\noindent
Under the same assumptions as Lemma \ref{lem:contcoer_left},
\begin{align}\label{eq:cont_coer_Dkinv1}
\left|\left\langle \matrixB^{(j)} \left(\matrixD_k^{(j)}\right)^{-1}\bv,\bw\right\rangle_{\left(\matrixD_k^{(j)}\right)^{-1}}\right| &\leq \Ccontj\N{\bv}_{\left(\matrixD_k^{(j)}\right)^{-1}}\N{\bw}_{\left(\matrixD_k^{(j)}\right)^{-1}},\quad\tand\\
\left|\left\langle \matrixB^{(j)}\left(\matrixD_k^{(j)}\right)^{-1} \bv,\bv\right\rangle_{\left(\matrixD_k^{(j)}\right)^{-1}}\right| &\geq \Ccoerj \N{\bv}^2_{\left(\matrixD_k^{(j)}\right)^{-1}}\label{eq:cont_coer_Dkinv2}
\end{align}
for all $\bv,\bw\in \Com^N$.
\end{lemma}

\begin{proof} The first equation in \eqref{eq:10} with $\widetilde{\bv}:= \matrixD_k^{(j)} \bv$ and $\widetilde{\bw}:= \matrixD_k^{(j)} \bw$ implies that 
$$
\left\langle \matrixB^{(j)} (\matrixD_k^{(j)})^{-1}\widetilde{\bv},\widetilde{\bw}\right\rangle_{\left(\matrixD_k^{(j)}\right)^{-1}}\leq \Ccontj\N{\left(\matrixD_k^{(j)}\right)^{-1}\widetilde{\bv}}_{\matrixD_k^{(j)}}\N{\left(\matrixD_k^{(j)}\right)^{-1}\widetilde{\bw}}_{\matrixD_k^{(j)}},
$$
and \eqref{eq:cont_coer_Dkinv1} follows since
$$
\N{\left(\matrixD_k^{(j)}\right)^{-1}\widetilde{\bv}}_{\matrixD^{(j)}_k}^2 = \N{\widetilde{\bv}}_{\left(\matrixD_k^{(j)}\right)^{-1}}^2
$$
from the definitions \eqref{eq:Dkip}. 
The proof of \eqref{eq:cont_coer_Dkinv2} is analogous.
\end{proof}

We now focus just on left preconditioning, but highlight that analogues of the results below hold for right preconditioning, using Lemma \ref{lem:contcoer_right} instead of Lemma \ref{lem:contcoer_left}.

Recall from Lemma \ref{lem:MScont} and Theorem \ref{thm:MScoer} that (with $\beta$ satisfying \eqref{eq:beta2} and \eqref{eq:beta1})
$$
\Ccont^{(1)} \sim k ,\quad 
\Ccont^{(2)}\sim 1, \quad
\Ccoer^{(1)}\sim 1, \quad 
\Ccoer^{(2)}\sim 1.
$$
These asymptotics imply that the ratio $\cos\sigma$ \eqref{eq:cosbeta} is independent of $k$ when $j=2$ (i.e.~when the $\matrixD_k^{(2)}$ matrix is used as a weight and $A=k^2$), but $\cos\sigma \sim 1/k$ when $j=1$ (i.e.~when the $\matrixD_k^{(1)}$ matrix is used as a weight and $A=1/3$). 
Combining this with Corollary \ref{cor:Elman} shows that weighted GMRES applied to $(\matrixD_k^{(j)})^{-1} \matrixB^{(j)}$ (with the weight $\matrixD_k^{(j)}$) converges in
a number of iterations depending linearly on $k$ when $j=1$, and in a $k$-independent number of iterations when $j=2$; we state these results as the following two theorems. 

\begin{theorem}[$k$-dependent GMRES convergence for $(\matrixD_k^{(1)})^{-1} \matrixB^{(1)}$]\label{thm:GMRES_us1} \

\noi
Assume that $\Omega$ is star-shaped with respect to a ball, $\beta$ satisfies both \eqref{eq:beta2} and \eqref{eq:beta1}, and $A=1/3$.

Let $\bu^m$ denote the $m$th iterate of weighted GMRES applied to the system $\matrixB^{(1)}\bu=\bg^{(1)}$, left  preconditioned with $(\matrixD_k^{(1)})^{-1}$, i.e.~the residual $\br^m=(\matrixD_k^{(1)})^{-1}(\bg^{(1)}-\matrixB^{(1)}\bu^m)$ is minimised in
the norm induced by $\matrixD_k^{(1)}$.

\

\noi (i) Given $k_0>0$, there exists a $C_1>0$, dependent on $k_0$ but independent of $k$, such that, given $0<\delta<1$, if 
\begin{equation}\label{eq:m1}
m\geq C_1 k \log \left(\frac{12}{\delta}\right),
\end{equation}
then
\begin{equation}\label{eq:conv_est1}
\frac{\Vert \bfr^m \Vert_{\matrixD_k^{(1)}}} { \Vert \bfr^0 \Vert_{\matrixD_k^{(1)}} } \
\leq  \delta
\end{equation}
for all $k\geq k_0$; i.e.~GMRES converges in a number of iterations at most linearly dependent on $k$.

\

\noi (ii) Moreover, let $u_N$ denote the Galerkin solution of the variational problem \eqref{eq:coercive} (i.e.~$u_N$ is the finite-element function corresponding to the vector $\bu$), and let $u_N^m$ denote the finite-element function corresponding to the $m$th iterate $\bu^m$. If the initial guess $\bu^0=\bze$, then for all $m$ satisfying \eqref{eq:m1},
\begin{equation}\label{eq:GMRES_rel_err1}
\frac{
\N{u_N-u_N^m}_{V_1}
}{
\N{u_N}_{V_1}
}
\leq \frac{
\Ccont^{(1)}
}{
\Ccoer^{(1)}
}\delta \sim k \delta.
\end{equation}
\end{theorem}

\begin{proof}
(i) From Lemma \ref{lem:contcoer_left}, Lemma \ref{lem:MScont}, and Theorem \ref{thm:MScoer} we have that 
$\Ccont^{(1)}\sim k$ and $\Ccoer^{(1)}\sim 1$, and thus $\cos \sigma$ defined by \eqref{eq:cosbeta} $\sim 1/k$. 
Since $\cos\sigma=\sin(\pi/2-\sigma) = (\pi/2-\sigma)(1+o(1))$ as $\sigma\rightarrow \pi/2$, we have that the variable $\epsilon$ in Corollary \ref{cor:Elman} $\sim 1/k$ and then the result of Part (i) follows.

(ii) To make the expressions more compact, we write $\matrixD_k$ for $\matrixD_k^{(1)}$ in this proof, and similarly for $\matrixB, \Ccont$, and $\Ccoer$. The residual-reduction bound \eqref{eq:conv_est1} with $\bu^0=\bze$ implies that
$$
\N{\matrixD_k^{-1} \matrixB(\bu^m -\bu) }_{\matrixD_k} \leq \delta \N{\matrixD_k^{-1}\matrixB\bu}_{\matrixD_k},
$$
so that
\begin{align*}
\N{\bu^m -\bu }_{\matrixD_k} &\leq \N{(\matrixD_k^{-1} \matrixB)^{-1}}_{\matrixD_k} \N{\matrixD_k^{-1} \matrixB(\bu^m-\bu)}_{\matrixD_k},\\
&\leq \N{(\matrixD_k^{-1} \matrixB)^{-1}}_{\matrixD_k} \delta \N{\matrixD_k^{-1} \matrixB \bu}_{\matrixD_k},\\
&\leq \N{(\matrixD_k^{-1} \matrixB)^{-1}}_{\matrixD_k} \delta \N{\matrixD_k^{-1} \matrixB}_{\matrixD_k}\N{\bu}_{\matrixD_k}\leq \frac{\Ccont}{\Ccoer} \delta \N{\bu}_{\matrixD_k} = \frac{\Ccont}{\Ccoer} \delta \N{u_N}_{V_1},
\end{align*}
where we have used both the norm equivalence \eqref{eq:normequivDk} and the facts that 
$$
\N{\matrixD_k^{-1} \matrixB}_{\matrixD_k}\leq \Ccont \quad\tand\quad \N{(\matrixD_k^{-1} \matrixB)^{-1}}_{\matrixD_k} \leq \frac{1}{\Ccoer},
$$
which follow from \eqref{eq:cont_coer_Dk}.

\end{proof}

\begin{theorem}[$k$-independent GMRES convergence for $(\matrixD_k^{(2)})^{-1} \matrixB^{(2)}$]\label{thm:GMRES_us2}
\quad Assume that $\Omega$ is star-shaped with respect to a ball, $\beta$ satisfies both \eqref{eq:beta2} and \eqref{eq:beta1}, and $A=k^2$.

Let $\bu^m$ denote the $m$th iterate of weighted GMRES applied to the system $\matrixB^{(2)}\bu=\bg^{(2)}$, left  preconditioned with $(\matrixD_k^{(2)})^{-1}$, i.e.~the residual $\br^m=(\matrixD_k^{(2)})^{-1}(\bg^{(2)}-\matrixB^{(2)}\bu^m)$ is minimised in
the norm induced by $\matrixD_k^{(2)}$.

\

\noi (i) 
Given $k_0>0$, there exists a $C_2>0$, dependent on $k_0$ but independent of $k$, such that, given $0<\delta<1$, if 
\begin{equation}\label{eq:m2}
m\geq C_2 \log \left(\frac{12}{\delta}\right),
\end{equation}
then
\begin{equation*}
\frac{\Vert \bfr^m \Vert_{\matrixD_k^{(2)}}} { \Vert \bfr^0 \Vert_{\matrixD_k^{(2)}} } \
\leq  \delta
\end{equation*}
for all $k\geq k_0$;
i.e.~GMRES converges in a $k$-independent number of iterations.

\

\noi (ii) Moreover, let $u_N$ denote the Galerkin solution of the variational problem \eqref{eq:coercive} (i.e.~$u_N$ is the finite-element function corresponding to the vector $\bu$), and let $u_N^m$ denote the finite-element function corresponding to the $m$th iterate $\bu^m$. If the initial guess $\bu^0=\bze$, then for all $m$ satisfying \eqref{eq:m2}
\begin{equation*}
\frac{
\N{u_N-u_N^m}_{V_2}
}{
\N{u_N}_{V_2}
}
\leq \frac{
\Ccont^{(2)}
}{
\Ccoer^{(2)}
}\delta \sim \delta.
\end{equation*}
\end{theorem}

\begin{proof}
The proof is very similar to that of Theorem \ref{thm:GMRES_us1}, but we now have $\Ccont^{(2)}\sim 1$ and $\Ccoer^{(2)}\sim 1$; in particular, in Part (i), $\epsilon$ is now independent of $k$. 
\end{proof}

\

As highlighted at the beginning of this subsection, all the results for MS with $A=k^2$ also hold for LS formulation because the $k$-dependence of $\Ccont/\Ccoer$ is the same for the two formulations (see Lemma \ref{lem:LSCC} and Corollary \ref{cor:MS}).

\subsection{Numerical experiments}\label{sec:fast_num}

\begin{figure}[htb]
\includegraphics[width=73mm]{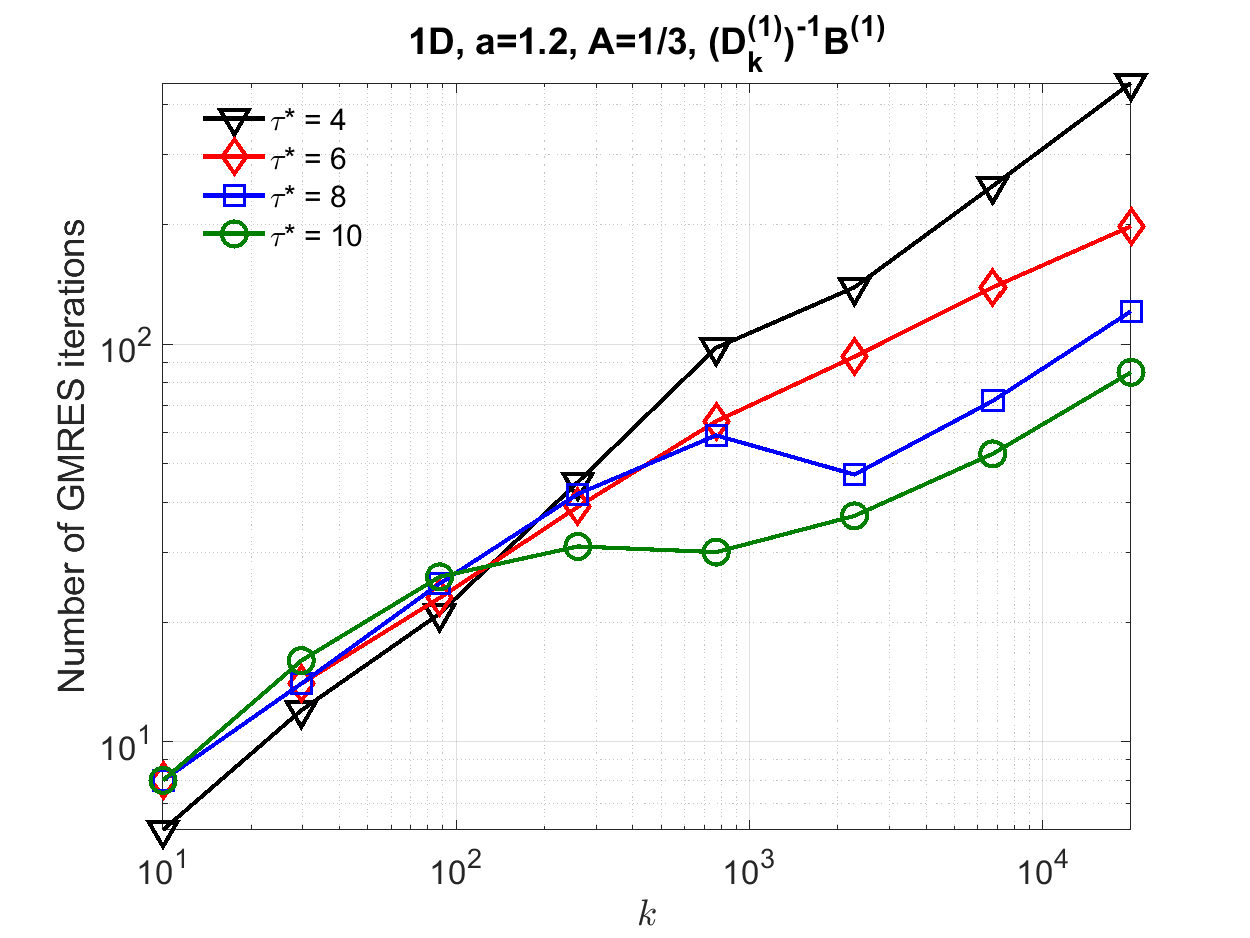}
\includegraphics[width=73mm]{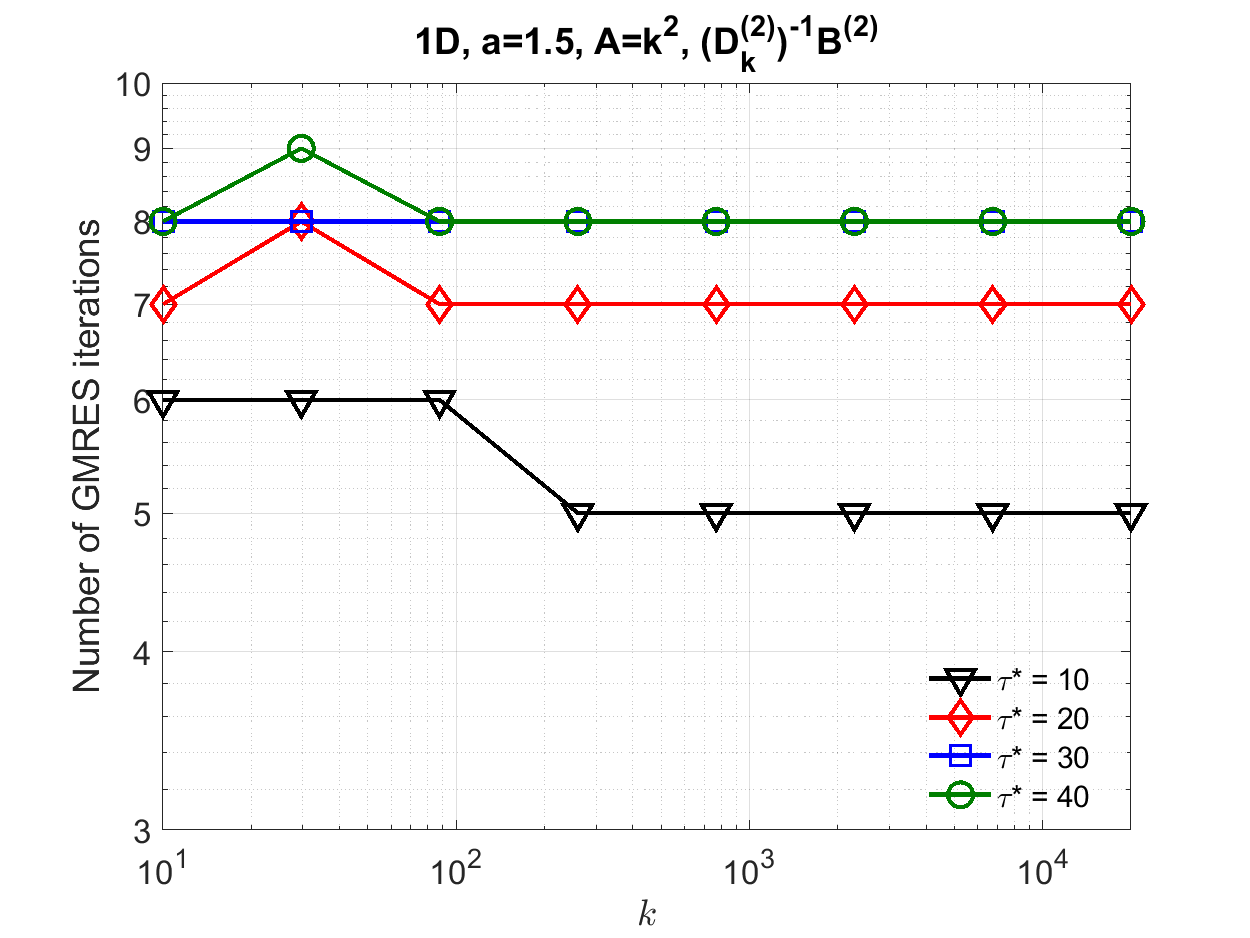}\\
\includegraphics[width=73mm]{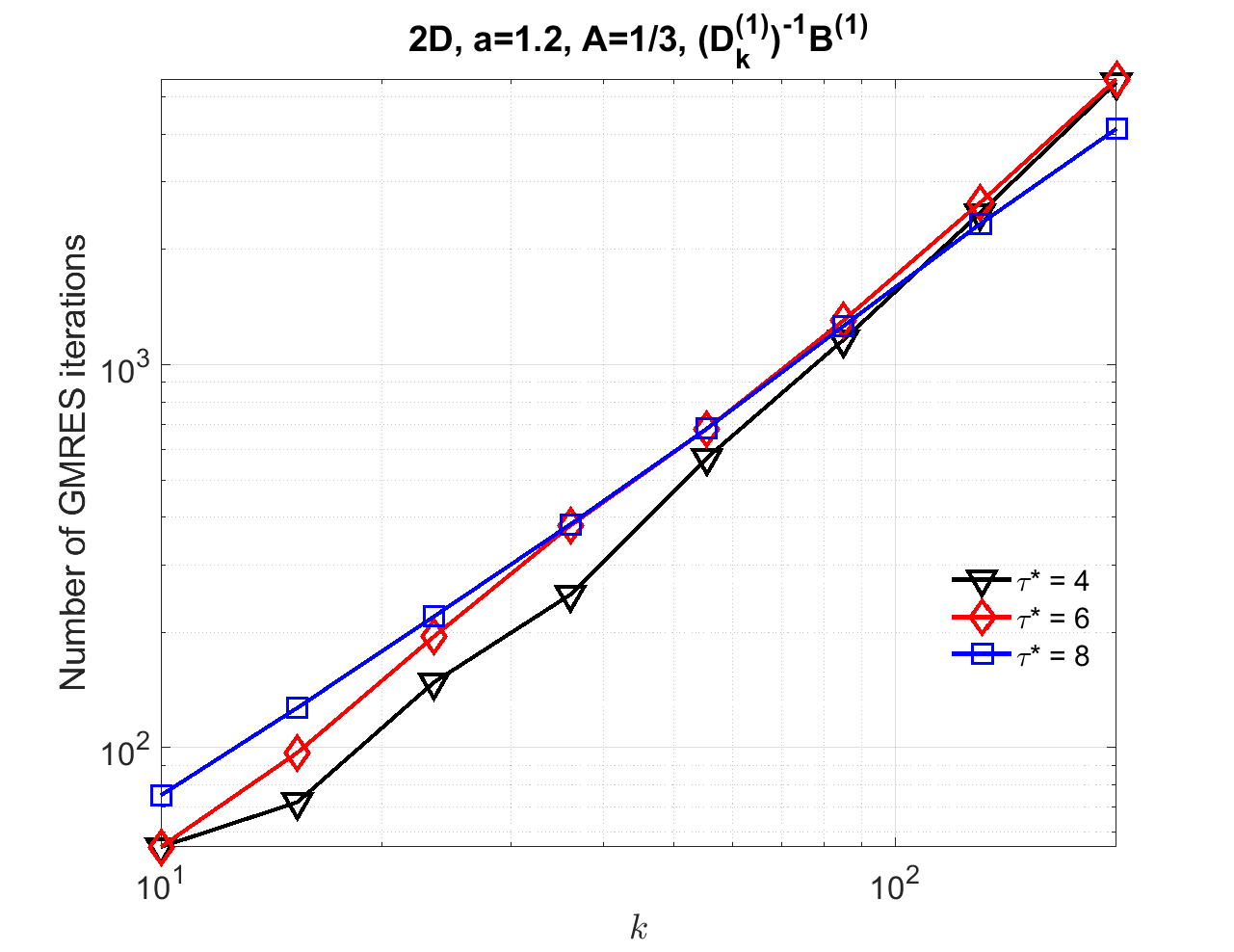}
\includegraphics[width=73mm]{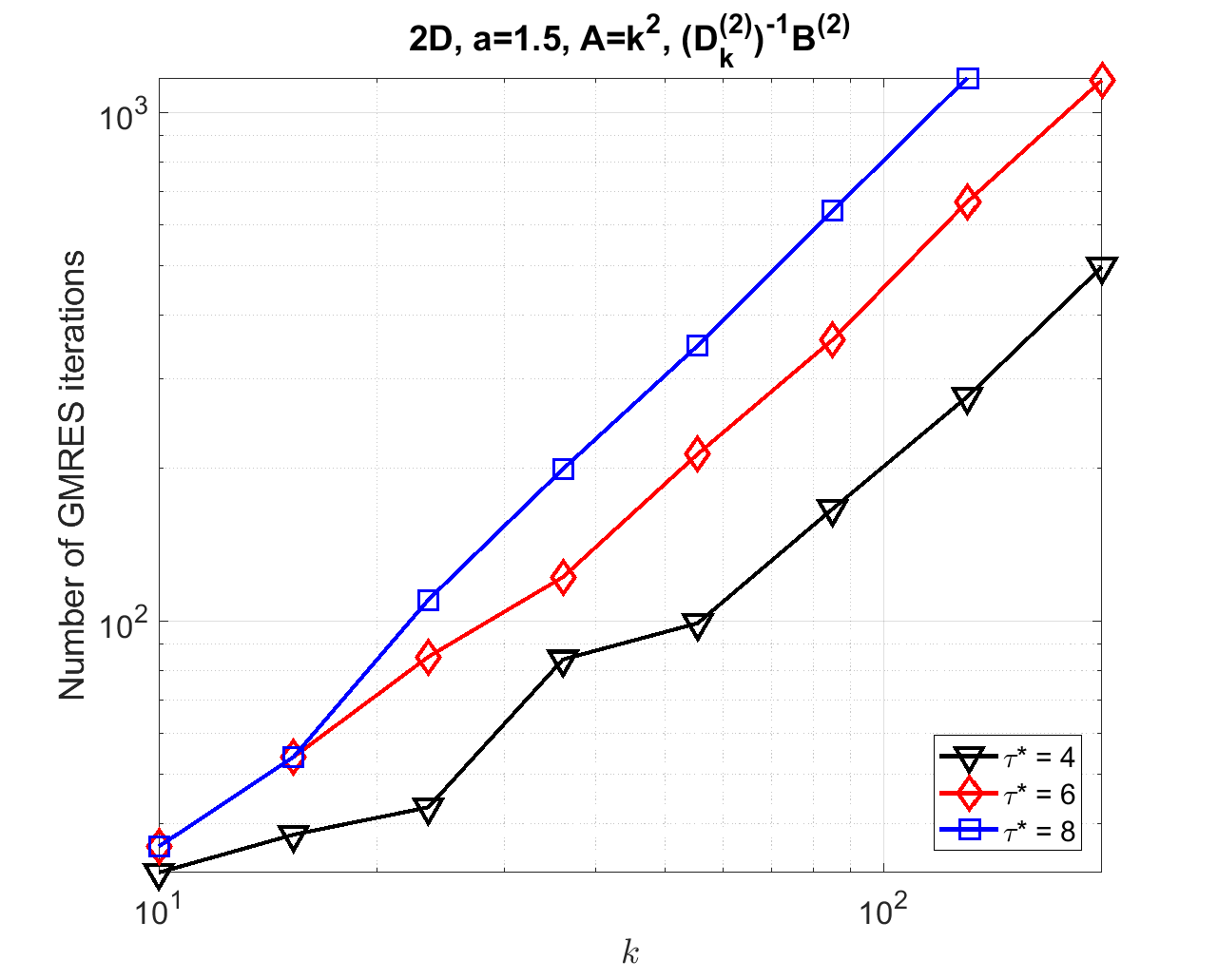}
\caption{The growth of the number of GMRES iterations with $k$, for both 1-d (top) and 2-d (bottom), for $(\matrixD_k^{(1)})^{-1} \matrixB^{(1)}$, discretised with $hk^{6/5}$ constant (left), and $(\matrixD_k^{(2)})^{-1} \matrixB^{(2)}$, discretised with $hk^{3/2}$ constant (right).
}
\label{fig:GMRES}
\end{figure}

We now describe experiments, in both 1- and 2-d, on the behaviour of GMRES applied to the MS formulation (involving the matrices $\matrixB^{(1)}$ \eqref{eq:Bg} and $\matrixB^{(2)}$ \eqref{eq:Bg2}), with the approximation space $V_N$ as described in \S\ref{sec:Vn}. We discuss these results in the context of preconditioning the standard formulation in \S\ref{sec:prec}.

Figure \ref{fig:GMRES} shows, for both 1- and 2-d, the growth of the number of GMRES iterations with $k$ for the situations described in Theorems \ref{thm:GMRES_us1} and \ref{thm:GMRES_us2} (i.e.~left preconditioning) except that Theorems \ref{thm:GMRES_us1} and \ref{thm:GMRES_us2} are for GMRES with weight $\matrixD_k^{(j)}$ ($j=1,2$)
applied to $(\matrixD_k^{(j)})^{-1}\matrixB^{(j)}$, and Figures \ref{fig:GMRES} display the results of standard (unweighted) GMRES applied to $(\matrixD_k^{(j)})^{-1}\matrixB^{(j)}$; we find the behaviour of weighted GMRES essentially identical  (note that this situation of the behaviour of GMRES being almost identical in the weighted and unweighted settings was also encountered in the domain-decomposition methods of \cite{BoDoGrSpTo:17, GrSpVa:17}). We also find 
essentially identical results for right preconditioning.

Although the results of Theorems \ref{thm:GMRES_us1} and \ref{thm:GMRES_us2} are independent of the mesh diameter $h$, in creating $ \matrixB^{(1)}$ (arising from \eqref{eq:B} with $A=1/3$) we choose $h$ such that $hk^{6/5}$ is constant, and in creating $\matrixB^{(2)}$ (arising from \eqref{eq:B} with $A=k^2$) we choose $h$ such that $hk^{3/2}$ is constant; recall that the theory and experiments in \S\ref{sec:test_acc} imply that, with these choices, the relative $H^1_k$-error in the Galerkin solutions is bounded independently of $k$.

The 1-d results are consistent with Theorems \ref{thm:GMRES_us1} and \ref{thm:GMRES_us2} in that the number of iterations for $(\matrixD_k^{(2)})^{-1} \matrixB^{(2)}$ is bounded independently of $k$, and the number of iterations for $(\matrixD_k^{(1)})^{-1} \matrixB^{(1)}$ grows at most linearly in $k$. In fact, in the latter case, the growth is sublinear: the rates (calculated from the least-squares best linear approximation) are $0.5692$, $0.4235$, $0.3227$, and $0.2583$ for $\tau^*=4,6,8,$ and $10$ respectively.

The rates of growth observed in the 2-d results are worse than predicted by Theorems \ref{thm:GMRES_us1} and \ref{thm:GMRES_us2}, although the growth for 
$(\matrixD_k^{(2)})^{-1} \matrixB^{(2)}$ is still less than that for $(\matrixD_k^{(1)})^{-1} \matrixB^{(1)}$. Indeed, 
for $\tau^*=4,6,8$, the rates of growth for 
$(\matrixD_k^{(1)})^{-1} \matrixB^{(1)}$ are $1.5829$, $1.5341$, and $1.3471$ respectively, and for $(\matrixD_k^{(2)})^{-1} \matrixB^{(2)}$ they are $0.9272$, $1.1608$, $1.3801$.

We have performed various checks to try to resolve the discrepancies between the numerical results and Theorems \ref{thm:GMRES_us1} and \ref{thm:GMRES_us2}. 
Although the theorems are valid for all $k\geq k_0$, and $k_0>0$ can be chosen arbitrarily small, since they are obtained using Corollary \ref{cor:Elman}, they will only be sharp in the limit $k\tendi$. Indeed, the $k\tendi$ limit corresponds to the $\epsilon\tendo$ limit in Corollary \ref{cor:Elman}, and any deviation from the large-$k$ asymptotics for $k$ small is then absorbed into the constants $C_1$ and $C_2$ in \eqref{eq:m1} and \eqref{eq:m2} respectively. It is therefore possible that the growth rates predicted by the theorems only manifest themselves for larger $k$ than considered in these numerical experiements.

\section{Estimates on the finite-element error of the GMRES solution}\label{sec:conc}

The results of Proposition \ref{lem:H1error} and Theorems \ref{thm:GMRES_us1}/\ref{thm:GMRES_us2} can be combined in the following theorems, but first we need to state an assumption about the solution of the IIP.

\begin{assumption}\label{ass:2}
The solution $u$ of the IIP satisfies
\begin{equation}\label{eq:ass2}
\N{u}_{V_1}\sim \N{u}_{V_2} \sim \N{u}_{\HokO}.
\end{equation}
\end{assumption}

Similar to Assumption \ref{ass:1} (discussed in Remark \ref{rem:ass1}), Assumption \ref{ass:2} is implicitly ruling out ``unphysical" $f$ and $g$.

\begin{theorem}[Summary of results about MS with $A=1/3$]\label{thm:main_result1}

Assume that $\Omega$ is star-shaped with respect to a ball, $\beta$ satisfies both \eqref{eq:beta2} and \eqref{eq:beta1}, and $A=1/3$. Assume further that Assumptions \ref{ass:1} and \ref{ass:2} are satisfied.

Let $\matrixB^{(1)}$ and $\bg^{(1)}$ be the Galerkin matrix and right-hand side, respectively, of the MS formulation \eqref{eq:Bg}, where the finite-dimensional subspace consists of  $C^1$ elements of fixed degree, and let $\matrixD_k^{(1)}$ be the symmetric, positive-definite 
matrix defined by \eqref{eq:Dk1}. Let $\bu$ be the solution of $\matrixB^{(1)}\bu =\bg^{(1)}$, so that the finite-element function corresponding to $\bu$ is $u_N$, the Galerkin solution.

Let $\bu^m$ denote the $m$th iterate of weighted GMRES, where the residual $\br^m$ is minimised in
the norm induced by $\matrixD_k^{(1)}$ (as described in \S\ref{sec:GMRES}), applied to the system $\matrixB^{(1)}\bu=\bg^{(1)}$, left  preconditioned with 
$(\matrixD_k^{(1)})^{-1} $. Let $u_N^m$ be the finite-element function corresponding to $\bu^m$.

Given $0<\epsilon<1$ and $k_0>0$, there exists a $C(\epsilon, k_0)$ and $m_0(\eps, k_0)$ (both independent of $h$ and $k$) such that 
\begin{equation}\label{eq:main_estimate1}
\text{if }\; hk^{7/4} \leq C(\epsilon, k_0)\; \tand \; m\geq m_0(\epsilon, k_0) k \log k\; \text{ then }\; 
\frac{\N{u-u_N^m}_{\HokO}}{\N{u}_{\HokO}} \leq \epsilon
\end{equation}
for all $k\geq k_0$.
\end{theorem}

\begin{theorem}[Summary of results about MS with $A=k^2$]\label{thm:main_result2}
Assume that $\Omega$, $\beta$, $\bu^m$, and $u_N^m$ are as in Theorem \ref{thm:main_result1}, except now that $A=k^2$ and we left precondition $\matrixB^{(2)}\bu=\bg^{(2)}$ with $(\matrixD_k^{(2)})^{-1} $.
Assume further that Assumptions \ref{ass:1} and \ref{ass:2} are satisfied.

Given $0<\epsilon<1$ and $k_0>0$, there exists a $C(\epsilon, k_0)$ and $m_0(\epsilon, k_0)$ (both independent of $h$ and $k$) such that 
\begin{equation}\label{eq:main_estimate2}
\text{if }\; hk^{3/2} \leq C(\epsilon, k_0)\; \tand \; m\geq m_0(\epsilon, k_0) \; \text{ then }\; 
\frac{\N{u-u_N^m}_{\HokO}}{\N{u}_{\HokO}}
 \leq \epsilon
\end{equation}
for all $k\geq k_0$.
\end{theorem}

Theorems \ref{thm:main_result1} shows that, under the condition $hk^{7/4}$ sufficiently small, an approximation to the solution of the IIP, with the error measured in the usual $\HokO$ norm, can be found in number of iterations growing at most like $k\log k$. 
We do not expect the $hk^{7/4}$ to be optimal since the experiments in \S\ref{sec:test_acc} indicated that this method is $hk^{6/5}$-accurate and thus Proposition \ref{lem:H1error} (which dictates the mesh threshold in Theorem \ref{thm:main_result1}) is not sharp.

Theorem \ref{thm:main_result2} shows that under the condition $hk^{3/2}$ sufficiently small, an approximation to the solution of the IIP, with the error measured in the usual $\HokO$ norm, can be found in $k$-independent number of iterations.

\

\begin{proof}[Proof of Theorem \ref{thm:main_result1}] 
In this proof we use $\Ccont$ and $\Ccoer$ to denote $\Ccont^{(1)}$ and $\Ccoer^{(1)}$, and we recall from Lemma \ref{lem:MScont} and Theorem \ref{thm:MScoer} that $\Ccont\sim k$ and $\Ccoer\sim 1$.
We combine the error bounds from Proposition \ref{lem:H1error} with the GMRES convergence estimates of Theorem \ref{thm:GMRES_us1}:
\begin{align*}
\N{u-u_N^m}_{H^1_k(\Omega)} &\le \N{u-u_N^m}_{V_1} \\
&\leq \N{u-u_N}_{V_1} + \N{u_N-u_N^m}_{V_1} \\
& \overset{\eqref{eq:3new1},\eqref{eq:GMRES_rel_err1}}\lesssim  h^2k^{7/2}(1+k^2h^2)\N{u}_{H^1_k(\Omega)} +  \frac{\Ccont}{\Ccoer}\delta \N{u_N}_{V_1}\\ 
&\lesssim  h^2k^{7/2}(1+k^2h^2)\N{u}_{H^1_k(\Omega)} +  \Big(\frac{\Ccont}{\Ccoer}\Big)^2\delta \N{u}_{V_1}.
\end{align*}
Here we also used
$\N{u_N}_{V_1}\le \Ccont\Ccoer^{-1}\N{u}_V$, which follows from Galerkin orthogonality
$\N{u_N}_{V_1}^2\le \Ccoer^{-1}B(u_N,u_N)=\Ccoer^{-1}B(u,u_N)\le \Ccont\Ccoer^{-1}\N{u}_{V_1}\N{u_N}_{V_1}$.
Choosing $\delta= \frac{\epsilon}{2}(\Ccoer/\Ccont)^2 \sim k^{-2}$ and $m\ge C_1 k\log(12/\delta)\sim k\log k$ and using $\N{u}_{V_1}\sim\N{u}_{H^1_k(\Omega)}$ from \eqref{eq:ass2}, we obtain \eqref{eq:main_estimate1}.
\end{proof}

\

\begin{proof}[Proof of Theorem \ref{thm:main_result2}] 
The proof is very similar to that of Theorem \ref{thm:main_result1}.
The only differences are in that $\Ccont/\Ccoer\sim1$ so $\delta\sim1$, the $\N{\cdot}_{V_1}$ norms are replaced by $\N{\cdot}_{V_2}$,
the error bound \eqref{eq:3new1} by \eqref{eq:3new2}, $h^2k^{7/2}$ by $h^2k^3$ and \eqref{eq:m1} by \eqref{eq:m2}.
\end{proof}

\section{Conclusions}

\subsection{Summary of the behaviour of the MS formulation}

When implementing the MS formulation \eqref{eq:coercive}, we have two choices: taking $A=1/3$ or $A=k^2$.
\begin{enumerate}
\item The MS formulation with $A=k^2$ behaves similarly to the least-squares formulation \eqref{eq:vfLS} in terms of accuracy -- both empirically and from Proposition \ref{lem:H1error} it is $hk^{3/2}$-accurate -- and this is worse than for the standard formulation \eqref{eq:vfH1}. On the other hand, we have a symmetric positive-definite preconditioner for this formulation
that (empirically) gives a $k$-independent number of iterations in 1-d and roughly linear growth in $k$ in 2-d.
This behaviour is summarised in Theorem \ref{thm:main_result2}, although this theorem also predicts a $k$-independent number of iterations in 2-d (at least when $k$ is sufficiently large), which is not borne out in the range of $k$ in the numerical experiments in \S\ref{sec:fast_num}.

\item The MS formulation with $A=1/3$ behaves similarly to the standard formulation in terms of accuracy: empirically we find it to be $hk^{6/5}$-accurate
We again have a symmetric positive-definite preconditioner that (empirically) gives growth ranging from $k^{0.25}$ to $k^{0.6}$ in 1-d, and from $k^{1.3}$ to $k^{1.6}$ in 2-d.
This behaviour is summarised in Theorem \ref{thm:main_result1}, although this theorem predicts $k\log k$ growth of the number of iterations for $d\geq 1$ (when $k$ is sufficiently large).
\end{enumerate}

Neither situation is ideal: with $A=k^2$, $k$-independent GMRES iterations are achieved with the preconditioner, but at the price of decreasing $h$ compared to the standard formulation (leading to a larger matrix). With $A=1/3$, $h$ can be chosen the same as for the standard formulation, but the preconditioner does not achieve the goal of having $k$-independent number of iterations.

\subsection{Discussion in the context of other work on preconditioning the Helmholtz equation}\label{sec:prec}

For specific geometries and decompositions, there now exist preconditioners for the Helmholtz equation that, at least empirically, (i) give a $k$-independent number of GMRES iterations, and (ii) can be computed in an efficient way. For example, the class of sequential domain-decomposition methods falling under the heading of ``sweeping" exhibit both these properties when applied in rectangular/cuboid geometries with tensor-product grids (see, e.g., the review \cite{GaZh:16} and the references therein), although the low-rank results that underlie these methods do not hold for general geometries and grids \cite{EnZh:18}.

As we saw in \S\ref{sec:GMRES_theorems}, the continuity and coercivity results of the MS formulation naturally give 
\begin{enumerate}
\item a symmetric, positive-definite preconditioner for the formulation (with the preconditioner depending on whether $A=1/3$ or $A=k^2$),
\item a rigorous bound on the number of the GMRES iterations, via the ``Elman estimate" \cite{El:82, EiElSc:83} and its improvement in \cite{BeGoTy:06}.
\end{enumerate}

Regarding 1: although the preconditioner does not give a $k$-independent number of GMRES iterations (except in 1-d with $A=k^2$), the fact that the preconditioner is a symmetric, positive-definite matrix with the same sparsity pattern of the Galerkin matrix allows one to apply solvers such as the conjugate gradient method.

Regarding 2: the only other rigorous bound in the literature on the number of GMRES iterations needed to solve a Helmholtz problem is in \cite{GaGrSp:15}. There, the authors prove that if the Galerkin matrix of the standard formulation of the IIP is preconditioned with the Galerkin matrix of the corresponding problem with absorption added in the form $\Delta +k^2 \mapsto \Delta + k^2 + \ri \eps$, then GMRES converges in a $k$-independent number of iterations when $\eps/k$ is sufficiently small. However, finding cheap approximations of the Galerkin matrix under this level of absorption is difficult; see \cite{GrSpVa:17}. Therefore, the MS formulation is currently the only formulation in the literature that has \emph{both} a symmetric, positive-definite preconditioner \emph{and} a rigorous bound on the number of GMRES iterations when applying the preconditioner (albeit with the number of iterations growing with $k$).

\subsection{Comparison with the results of \texorpdfstring{\cite{GaMo:17a, GaMo:17b}}{Ganesh and Morgenstern}}\label{rem:GM1}
In \cite{GaMo:17a}, Ganesh and Morgenstern discretise the MS formulation of Definition \ref{def:MS} above using $C^1$ finite-dimensional subspaces built from splines. 

The philosophy in \cite{GaMo:17a} is slightly different to ours: here we determined for which $a$ the Galerkin method is $hk^{a}$-accurate and then investigated solving the linear system for increasing $k$, with $h$ and $p$ chosen to ensure $k$-independent accuracy (so that the number of degrees of freedom increases with $k$).
In contrast, 
the majority of numerical experiments in \cite{GaMo:17a} are for fixed $h$ and $p$ (i.e~a fixed number of degrees of freedom) and increasing $k$, in which case the accuracy of the Galerkin solutions then decreases with $k$ (the exception are the experiments in \cite[\S4.2]{GaMo:17a} which demonstrate convergence as $h$ decreases for fixed $p$ and $k$, and as $p$ increases for fixed $h$ and $k$). We therefore cannot directly compare any of the results in \cite{GaMo:17a} to ours, but we now give a brief overview.

The first main goal in \cite{GaMo:17a} is to numerically optimise the parameter $\beta$ (for fixed, $h, p,$, and $k$) to minimise the number of GMRES iterations needed to solve the (unpreconditioned) system \cite[\S3.1]{GaMo:17a}, and then in \cite[\S3.2]{GaMo:17a} a formula is obtained for $\beta$ in terms of $h,p,$ and $k$ that provides a good approximation to the optimal $\beta$. \cite{GaMo:17a} then consider preconditioning the linear system by adding absorption and using additive Schwarz domain decomposition with Dirichlet boundary conditions on the subdomains.

The paper \cite{GaMo:17b} obtains the analogue of the formulation in Definition \ref{def:MS} for the interior impedance problem for the operator $\Delta +k^2 n$, and then proves that this formulation is coercive if the refractive index $n$ satisfies a condition that guarantees nontrapping of rays (see \cite[\S6]{GrPeSp:18}). Numerical experiments demonstrating the convergence of the Galerkin solutions as $h$ decreases for fixed $p$ and $k$, and as $p$ increases for fixed $h$ and $k$, and with a particular emphasis on non-smooth solutions and solutions in non-starshaped domains, are given in \cite[\S\S5.1, 5.2, 7.1, 7.2]{GaMo:17b}. The convergence of preconditioned GMRES is then investigated as in \cite{GaMo:17a}: for $h$ and $p$ fixed and increasing $k$, and using the same preconditioner.

\subsection{Concluding remarks}

\noi The lack of coercivity of the standard variational formulation of the Helmholtz equation is often cited as one of the reasons the Helmholtz equation is difficult to solve numerically; for example, the following is the first line of \cite{CoGa:18}
\begin{quotation}
\noindent``Solving discretized Helmholtz problems by iterative methods is challenging, mainly because of the lack of coercivity of the continuous operator and the highly oscillatory nature of the solutions."
\end{quotation}
and Point 3 on the first page of the present paper expresses the same ``lack of coercivity" sentiment in slightly more detail.

The results of the present paper show that the situation is more subtle (and not captured just by ``lack of coercivity"); for example, both the standard formulation and the MS formulation with $A=1/3$ suffer from the pollution effect in the same way (as shown in the investigation of Q1 in \S\ref{sec:test_acc}), but the MS formulation does not provide an immediate fix to the problems facing iterative methods (as show in the investigation of Q2 in \S\ref{sec:test_fast}).

\paragraph{Acknowledgements.}
We thank Melina Freitag (University of Bath) and Jennifer Pestana (University of Strathclyde) for useful discussions and for sharing their GMRES codes with us. We thank Simon Chandler-Wilde (University of Reading), Mahadevan Ganesh (Colorado School of Mines), and Alastair Spence (University of Bath) for useful discussions.
We also thank the referees for constructive comments.
G.D.~and A.M.~were supported by EPSRC grant EP/N019407/1, and E.A.S.~was supported by EPSRC grant  EP/R005591/1.

\section*{References}

\end{document}